\documentclass{amsart}

\usepackage{amsfonts,amssymb,verbatim,amsmath,amsthm,latexsym,textcomp,amscd}
\usepackage{latexsym,amsfonts,amssymb,epsfig,verbatim}
\usepackage{amsmath,amsthm,amssymb,latexsym,graphics,textcomp, hyperref}
\usepackage{graphicx}
\usepackage{color}
\usepackage{url}

\begin{document}

\newtheorem{theorem}{Theorem}[section]
\newtheorem{prop}[theorem]{Proposition}
\newtheorem{lemma}[theorem]{Lemma}
\newtheorem{cor}[theorem]{Corollary}
\newtheorem{definition}[theorem]{Definition}
\newtheorem{defn}[theorem]{Definition}
\newtheorem{conj}[theorem]{Conjecture}
\newtheorem{claim}[theorem]{Claim}
\newtheorem{defth}[theorem]{Definition-Theorem}
\newtheorem{obs}[theorem]{Observation}
\newtheorem{rmk}[theorem]{Remark}
\newtheorem{remark}[theorem]{Remark}
\newtheorem{qn}[theorem]{Question}

\newcommand{\hhat}{\widehat}
\newcommand{\boundary}{\partial}
\newcommand{\C}{{\mathbb C}}
\newcommand{\integers}{{\mathbb Z}}
\newcommand{\natls}{{\mathbb N}}
\newcommand{\ratls}{{\mathbb Q}}
\newcommand{\reals}{{\mathbb R}}
\newcommand{\proj}{{\mathbb P}}
\newcommand{\lhp}{{\mathbb L}}
\newcommand{\tube}{{\mathbb T}}
\newcommand{\cusp}{{\mathbb P}}
\newcommand\AAA{{\mathcal A}}
\newcommand\BB{{\mathcal B}}
\newcommand\CC{{\mathcal C}}
\newcommand\DD{{\mathcal D}}
\newcommand\EE{{\mathcal E}}
\newcommand\FF{{\mathcal F}}
\newcommand\GG{{\mathcal G}}
\newcommand\HH{{\mathcal H}}
\newcommand\II{{\mathcal I}}
\newcommand\JJ{{\mathcal J}}
\newcommand\KK{{\mathcal K}}
\newcommand\LL{{\mathcal L}}
\newcommand\MM{{\mathcal M}}
\newcommand\NN{{\mathcal N}}
\newcommand\OO{{\mathcal O}}
\newcommand\PP{{\mathcal P}}
\newcommand\QQ{{\mathcal Q}}
\newcommand\RR{{\mathcal R}}
\newcommand\SSS{{\mathcal S}}
\newcommand\TT{{\mathcal T}}
\newcommand\UU{{\mathcal U}}
\newcommand\VV{{\mathcal V}}
\newcommand\WW{{\mathcal W}}
\newcommand\XX{{\mathcal X}}
\newcommand\YY{{\mathcal Y}}
\newcommand\ZZ{{\mathcal Z}}
\newcommand\CH{{\CC\Hyp}}
\newcommand\MF{{\MM\FF}}
\newcommand\PMF{{\PP\kern-2pt\MM\FF}}
\newcommand\ML{{\MM\LL}}
\newcommand\PML{{\PP\kern-2pt\MM\LL}}
\newcommand\GL{{\GG\LL}}
\newcommand\Pol{{\mathcal P}}
\newcommand\half{{\textstyle{\frac12}}}
\newcommand\Half{{\frac12}}
\newcommand\Mod{\operatorname{Mod}}
\newcommand\Area{\operatorname{Area}}
\newcommand\ep{\epsilon}
\newcommand\Hypat{\widehat}
\newcommand\Proj{{\mathbf P}}
\newcommand\U{{\mathbf U}}
 \newcommand\Hyp{{\mathbf H}}
\newcommand\D{{\mathbf D}}
\newcommand\Z{{\mathbb Z}}
\newcommand\R{{\mathbb R}}
\newcommand\Q{{\mathbb Q}}
\newcommand\E{{\mathbb E}}
\newcommand\til{\widetilde}
\newcommand\length{\operatorname{length}}
\newcommand\tr{\operatorname{tr}}
\newcommand\gesim{\succ}
\newcommand\lesim{\prec}
\newcommand\simle{\lesim}
\newcommand\simge{\gesim}
\newcommand{\simmult}{\asymp}
\newcommand{\simadd}{\mathrel{\overset{\text{\tiny $+$}}{\sim}}}
\newcommand{\ssm}{\setminus}
\newcommand{\diam}{\operatorname{diam}}
\newcommand{\pair}[1]{\langle #1\rangle}
\newcommand{\T}{{\mathbf T}}
\newcommand{\inj}{\operatorname{inj}}
\newcommand{\pleat}{\operatorname{\mathbf{pleat}}}
\newcommand{\short}{\operatorname{\mathbf{short}}}
\newcommand{\vertices}{\operatorname{vert}}
\newcommand{\collar}{\operatorname{\mathbf{collar}}}
\newcommand{\bcollar}{\operatorname{\overline{\mathbf{collar}}}}
\newcommand{\I}{{\mathbf I}}
\newcommand{\tprec}{\prec_t}
\newcommand{\fprec}{\prec_f}
\newcommand{\bprec}{\prec_b}
\newcommand{\pprec}{\prec_p}
\newcommand{\ppreceq}{\preceq_p}
\newcommand{\sprec}{\prec_s}
\newcommand{\cpreceq}{\preceq_c}
\newcommand{\cprec}{\prec_c}
\newcommand{\topprec}{\prec_{\rm top}}
\newcommand{\Topprec}{\prec_{\rm TOP}}
\newcommand{\fsub}{\mathrel{\scriptstyle\searrow}}
\newcommand{\bsub}{\mathrel{\scriptstyle\swarrow}}
\newcommand{\fsubd}{\mathrel{{\scriptstyle\searrow}\kern-1ex^d\kern0.5ex}}
\newcommand{\bsubd}{\mathrel{{\scriptstyle\swarrow}\kern-1.6ex^d\kern0.8ex}}
\newcommand{\fsubeq}{\mathrel{\raise-.7ex\hbox{$\overset{\searrow}{=}$}}}
\newcommand{\bsubeq}{\mathrel{\raise-.7ex\hbox{$\overset{\swarrow}{=}$}}}
\newcommand{\tw}{\operatorname{tw}}
\newcommand{\base}{\operatorname{base}}
\newcommand{\trans}{\operatorname{trans}}
\newcommand{\rest}{|_}
\newcommand{\bbar}{\overline}
\newcommand{\UML}{\operatorname{\UU\MM\LL}}
\newcommand{\EL}{\mathcal{EL}}
\newcommand{\tsum}{\sideset{}{'}\sum}
\newcommand{\tsh}[1]{\left\{\kern-.9ex\left\{#1\right\}\kern-.9ex\right\}}
\newcommand{\Tsh}[2]{\tsh{#2}_{#1}}
\newcommand{\qeq}{\mathrel{\approx}}
\newcommand{\Qeq}[1]{\mathrel{\approx_{#1}}}
\newcommand{\qle}{\lesssim}
\newcommand{\Qle}[1]{\mathrel{\lesssim_{#1}}}
\newcommand{\simp}{\operatorname{simp}}
\newcommand{\vsucc}{\operatorname{succ}}
\newcommand{\vpred}{\operatorname{pred}}
\newcommand\fhalf[1]{\overrightarrow {#1}}
\newcommand\bhalf[1]{\overleftarrow {#1}}
\newcommand\sleft{_{\text{left}}}
\newcommand\sright{_{\text{right}}}
\newcommand\sbtop{_{\text{top}}}
\newcommand\sbot{_{\text{bot}}}
\newcommand\sll{_{\mathbf l}}
\newcommand\srr{_{\mathbf r}}
\newcommand\geod{\operatorname{\mathbf g}}
\newcommand\mtorus[1]{\boundary U(#1)}
\newcommand\A{\mathbf A}
\newcommand\Aleft[1]{\A\sleft(#1)}
\newcommand\Aright[1]{\A\sright(#1)}
\newcommand\Atop[1]{\A\sbtop(#1)}
\newcommand\Abot[1]{\A\sbot(#1)}
\newcommand\boundvert{{\boundary_{||}}}
\newcommand\storus[1]{U(#1)}
\newcommand\Momega{\omega_M}
\newcommand\nomega{\omega_\nu}
\newcommand\twist{\operatorname{tw}}
\newcommand\modl{M_\nu}
\newcommand\MT{{\mathbb T}}
\newcommand\Teich{{\mathcal T}}
\renewcommand{\Re}{\operatorname{Re}}
\renewcommand{\Im}{\operatorname{Im}}

\title{Cannon-Thurston Maps for Kleinian Groups}

\author{Mahan Mj}

\address{School of Mathematics, Tata Institute of Fundamental Research, Homi Bhabha Road, Mumbai-400005, India}
\email{mahan.mj@gmail.com; mahan@math.tifr.res.in}

\begin{abstract}
 We
 show that Cannon-Thurston maps  exist for  degenerate free groups without parabolics, i.e. for handlebody groups.
Combining these techniques  with earlier work proving the existence of Cannon-Thurston maps for surface groups, we show 
that Cannon-Thurston maps  exist for
arbitrary finitely generated Kleinian groups without parabolics,
proving  conjectures of Thurston and McMullen. We also show that point pre-images under Cannon-Thurston maps for  degenerate
free groups without parabolics
correspond to
end-points of leaves of an ending lamination in the Masur domain,  whenever a point has more than one pre-image. This proves a conjecture of Otal. We also prove a similar result for 
 point pre-images under Cannon-Thurston maps for arbitrary finitely generated Kleinian groups  without parabolics.

\end{abstract}

\subjclass[2010]{57M50, 20F67 (Primary); 20F65,  22E40  (Secondary)}

\maketitle

 \hfill
To Bill Thurston for lasting inspiration.

\tableofcontents

\section{Introduction} Let $G$ be a Kleinian group, $L_G$ its limit set in the boundary sphere $S^2$, and $D_G (= S^2 \setminus L_G)$
its domain of discontinuity. We paraphrase Question 14 of Thurston's problem-list \cite{thurston-bams} below:

\begin{qn} Suppose that $\Gamma$ is a geometrically finite Kleinian group and $G$ an arbitrary Kleinian group abstractly isomorphic to $\Gamma$
via a weakly type-preserving isomorphism, i.e. an isomorphism taking parabolics of $\Gamma$ to parabolics of $G$. 
Then
is it true that there is a continuous map $g$ from the  limit set $L_\Gamma$ (of $\Gamma$) onto the limit set $L_G$ (of $G$)
taking the fixed point(s) of an element $\gamma$ to the fixed point(s) of the corresponding element  $\gamma^\prime$? \label{ctqn} \end{qn}

A continuous map $g$ as in Question \ref{ctqn} is called a {\bf Cannon-Thurston map} because of Cannon and Thurston's seminal paper  \cite{CT, CTpub}.
We  refer to the Introduction of \cite{mahan-split}  for  a detailed history of the problem and mention only that
in \cite{mahan-split} we showed that simply or doubly degenerate surface Kleinian groups without accidental parabolics admit Cannon-Thurston maps.
 In \cite{mahan-elct} we had shown that point pre-images 
of the Cannon-Thurston map for  simply or doubly degenerate groups without accidental parabolics
 correspond to endpoints of leaves of ending laminations whenever a point has more than one pre-image. The main aim of this paper is to apply
the techniques developed in  \cite{mahan-split} and    \cite{mahan-elct} to extend these results to arbitrary finitely generated
Kleinian groups without parabolics (and more generally to manifolds whose ends
admit a Minsky model), thus answering affirmatively Question \ref{ctqn} as also questions of 
McMullen \cite{ctm-locconn} and Otal \cite{otal-thesis}. This completes the project starting with \cite{mahan-split} and proceeding through \cite{mahan-elct,mahan-red}.\footnote{An earlier version of
	some parts of this paper existed in draft form in an earlier version of \cite{mahan-split}. The division of material between \cite{mahan-split} and the present paper is in the interests of readability.}

One more ingredient is necessary before we proceed with  statements of the main results.
A geodesic lamination on a hyperbolic surface is a foliation of a closed subset by geodesics. Let $E$ be a degenerate end 
of a hyperbolic 3-manifold with boundary surface $S = \partial E$. Then there exists a sequence of simple closed curves $\{ \sigma_n \}$ on $S$
whose geodesic realizations exit the end $E$
(this innocent sounding statement is a consequence of deep work of several authors including Thurston \cite{thurstonnotes},
Bonahon \cite{bonahon-bouts}, Canary \cite{canary}, Agol \cite{agol-tameness}
and Calegari-Gabai \cite{gab-cal}). 
Then the limit of such a sequence (in the space of projectivized measured laminations $\PML (S)$; for the time being, the reader
will not be too far off if (s)he thinks of the Hausdorff limit on the bounding surface $S$ of $E$) is a lamination $\lambda$. It turns out that
$\lambda$ is independent of the sequence $\{ \sigma_n \}$ and is called the {\bf ending lamination} for the end $E$.
 The following provides one of the main theorems of this paper. \\
\noindent { \bf Theorem \ref{ptpreimagefinal}} {\it Let $G$ be a finitely generated free degenerate  Kleinian group without parabolics. Let $i : \Gamma_G \rightarrow {\mathbb H}^3$ be the natural identification
of a Cayley graph of $G$ with the orbit of a point in ${\mathbb H}^3$. Then $i$ extends continuously to a map 
$\hat{i}: \hhat{\Gamma_G} \rightarrow {\mathbb D}^3$ between the compactifications $\hhat{\Gamma_G}, {\mathbb D}^3$
of ${\Gamma_G}, {\mathbb H}^3$ respectively. Let $\partial i$ denote the restriction of $\hat{i}$ to the boundary $\partial \Gamma_G$ of $\Gamma_G$.
Then $\partial i(a) = \partial i(b)$ for $a \neq b \in \partial \Gamma_G$  if and only if 
  $a, b$ are either ideal 
end-points of a leaf of an ending lamination  of $G$, or ideal boundary points of a
complementary ideal polygon.  }

 \smallskip
 
The hyperbolic boundary $\partial \Gamma_G$ is ($G-$equivariantly) homeomorphic to the limit set of some (any) geometrically finite
group (without parabolics) isomorphic to $G$. Hence Theorem \ref{ptpreimagefinal} above provides a positive answer to Question \ref{ctqn}
for free degenerate  Kleinian group without parabolics. 
There are three main new ingredients of the proof over and above \cite{mahan-split}:
\begin{enumerate}
\item[a)] We need to show that split components sufficiently deep inside an end
are incompressible (Proposition \ref{incompressible}). This was automatic in \cite{mahan-split}.
\item[b)]
A crucial idea in the proof of  the existence of the Cannon-Thurston map $\hat i$ in Theorem \ref{ptpreimagefinal} goes back to 
work of Miyachi \cite{miyachi-ct} (see also \cite{souto-ct})  in the case of bounded geometry. There a collection of disjoint
embedded disks are constructed in the (compact) core handlebody cutting it up into a ball. The boundary circles of these disks are flowed
out the end giving rise to a collection of quasiconvex disks. 
\item[c)] The broad idea here is similar to (b).
The basic difference between Miyachi's approach and ours is that we are forced to use a coarse model
rather than a continuous one, forcing the methods of this paper to be technically quite a bit more involved. To 
tackle this issue we need to use the split geometry model \cite{mahan-split}, recalled in Section \ref{min}, 
 and introduce certain "admissible paths" in Section
\ref{sec-adm}.
\end{enumerate}

\smallskip

The proof of Theorem \ref{ptpreimagefinal} generalizes with some modifications to arbitrary finitely generated Kleinian groups.
The relative 
hyperbolic boundary $\partial \Gamma_G$ \cite{bowditch-relhyp}
of a Kleinian group $G$  is ($G-$equivariantly) homeomorphic to the limit set of some (any) geometrically finite
 Kleinian group isomorphic to $G$, provided the isomorphism is strictly type-preserving (i.e. it maps parabolics to parabolics and pulls back 
parabolics to parabolics). 

\medskip

\noindent { \bf Theorems \ref{ct} and \ref{ptpreimagefinal-kg}} {\it Let $G$ be a finitely generated Kleinian group. 
 Let $i : \Gamma_G \rightarrow {\mathbb H}^3$ be the natural identification
of a Cayley graph of $G$ with the orbit of a point in ${\mathbb H}^3$. Let $M= {\mathbb H}^3/G$ and assume that each degenerate end $E$
of $M$ admits a bi-Lipschitz Minsky model (for instance if $M$ has no parabolics, see Remark \ref{minapp} below). 
 Then $i$ extends continuously to a map 
$\hat{i}: \hhat{\Gamma_G} \rightarrow {\mathbb D}^3$,
where $\hhat{\Gamma_G}$ denotes the (relative) hyperbolic compactification of $\Gamma_G$. Let $\partial i$ denote the restriction of $\hat{i}$ to the boundary $\partial \Gamma_G$ of $\Gamma_G$.

Let $E$ be a degenerate end of $N^h= {\mathbb H}^3/G$ and $\til E$ a lift of $E$ to $\til{N^h}$
and let $M_{gf}$ be an augmented Scott core of $N^h$. Then the ending lamination $\LL_E$ for the end $E$ lifts to a lamination
on $\til{M_{gf}} \cap \til{E}$. Each such lift $\LL$ of the ending lamination of a degenerate end defines a relation $\RR_\LL$ on the (Gromov) 
hyperbolic boundary $ \partial \widetilde{M_{gf}}$
(equal to the relative hyperbolic boundary $\partial \Gamma_G$ of $\Gamma_G$),
 given by
$a\RR_\LL b$ if and only if
 $a, b$ are  end-points of a leaf of $\LL$. Let $\{ \RR_i \}_i$ be the entire collection of relations on $ \partial \widetilde{M_{gf}}$ obtained this way. Let $\RR$ be the transitive closure of
the union $\bigcup_i \RR_i$. Then $\partial i(a) = \partial i(b)$ if and only if $a\RR b$.} 

\smallskip

\begin{remark}\label{minapp}
Theorem \ref{ct} gives an affirmative answer to a conjecture of McMullen \cite{ctm-locconn} and Theorem \ref{ptpreimagefinal}
gives an affirmative answer to a conjecture of Otal \cite{otal-thesis} under the assumption of the existence of a Minsky model.
The existence of
such a model was established in \cite{minsky-elc1, minsky-elc2} for manifolds with incompressible boundary and announced
for the general (compressible boundary)
 case in \cite{minsky-elc3}. Since \cite{minsky-elc3} has not yet appeared we give a sketch of a proof of the existence of a Minsky
model in the special case that $M$ has no parabolics
following ideas of Brock, Bromberg and Souto in the Appendix \ref{app}. 
\end{remark}

 For ease of exposition, throughout this paper, we shall often first work out the problem for free groups and then indicate the 
generalization to arbitrary finitely generated Kleinian groups.

\smallskip

\noindent {\bf Acknowledgments:} I am grateful to Jean-Pierre Otal for suggesting the problem of finding point pre-images of the Cannon-Thurston
map for handlebodies; and for giving me a copy of his thesis \cite{otal-thesis}, where the structure of Cannon-Thurston maps for handlebody groups
is conjectured. I thank Ken Bromberg  for explaining
the grafting construction in \cite{bromberg-proj} and the proof in the Appendix to me. 
I would  also like to thank the referee for  detailed  suggestions and corrections. Research of the author is supported in part   by  CEFIPRA Indo-French Research grant 4301-1 
	and in part by a DST JC Bose Fellowship.

\subsection{Outline of Paper  and Scheme of Proof} After discussing some preliminary material on Relative Hyperbolicity and Cannon-Thurston maps in
Section \ref{prel}, we recall the essential technical tools from \cite{mahan-split} in Section \ref{min}. 
Section \ref{min} has to do with the fact that any end admits a {\it split geometry}
structure.  Consider a geometrically infinite end $E$ homeomorphic to $S \times [0,\infty)$ where $S$ is a compact 
hyperbolic surface (possibly with boundary).
Split geometry roughly
gives a sequence of embedded surfaces $\{ \Sigma_i \}$ exiting $E$, such that successive surfaces $\Sigma_i, \Sigma_{i+1}$ bound between them a block $B_i$,
which either has bounded geometry, (i.e. is uniformly bi-Lipschitz homeomorphic to $S \times [0,1]$ where the latter has the product metric)
or contains a thin Margulis tube running  vertically from $\Sigma_i$ to $\Sigma_{i+1}$ and `splitting' $B_i$ into split components. Blocks of the latter kind are
called  {\it split blocks}. Electrocuting the lifts of split components in the universal cover
 gives rise to a combinatorial metric $d_G$ called a {\it graph metric} on $\til E$. 

Section \ref{mainsec} is the core of the paper and proves the existence of Cannon-Thurston maps for arbitrary finitely generated Kleinian groups $G$ under the additional assumption that each degenerate end of $M$ admits a Minsky model.
Modulo \cite{mahan-split, mahan-red} the proof reduces to proving this for manifolds $M = {\mathbb{H}}^3/G$ with compressible core.
The prototypical case is that of free degenerate Kleinian groups without parabolics, which is what we elaborate on here in the Introduction. 

 We briefly recall the proof of Miyachi \cite{miyachi-ct}
(or Souto \cite{souto-ct}) in the case of bounded geometry free degenerate Kleinian groups without parabolics. Let $H$ denote the
handlebody compact core of $M$. The end $E (=M \setminus H)$ is equipped with a Sol-type
metric as in \cite{CTpub}. A finite collection of disjoint disks $D_1, \cdots D_k$ on $H$ are chosen cutting $H$ up into a ball.
Let $\sigma_i$ denote the boundary curve of $D_i$. Then $\sigma_i$ can be canonically `flowed' out the end $E$ using the metric on $E$. In fact
$E$ may be thought of as (bi-Lipschitz homeomorphic to) the universal curve over a Teichmuller geodesic ray $\eta$ and each $\sigma_i$
gives rise to an annulus $C_i$ -- the `flow image' of $\sigma_i$ by the flow given by the one-parameter family of Teichmuller maps along $\eta$.
The collection $\{ A_i (=D_i \cup C_i)\}$ cut $M$ up into a (non-compact) ball. The crucial point in \cite{miyachi-ct} (or \cite{souto-ct}) that makes their proof work
is the quasiconvexity of each $A_i$. This is proved there using techniques similar to \cite{CTpub}. Once this is done, it follows more or less
automatically
that if a geodesic segment {\em in $\til H$} joins a pair of points $a,b$
 in a complementary component of a lift $D$ (of one of the $D_i$'s) then the geodesic
{\em in $\til M$} joining $a,b$ lies (coarsely) in the corresponding component of $(\til M \setminus A)$, where $A$ is the lift of $A_i$ containing $D$.

In this paper, we do not have the luxury of flowing along a Teichmuller geodesic. Instead we use the split geometry model of the end $E$
to 
\begin{enumerate}
\item discretize the problem by using the level surfaces $\Sigma_i$ given by split geometry.
\item construct a discretized flow image, by considering a `coarse annulus' replacement of $A_i$ by taking
the union  (over $i$) of closed geodesics on $\Sigma_i$
in the same homotopy class as $\sigma_i$. 
\end{enumerate}

 We need to work in $\til M$ rather than $\til E$. Hence, in order to carry over the split geometry machinery (especially the crucially important graph metric $d_G$)
 in the context of compressible cores, we need to  ensure that split components are actually incompressible sufficiently deep into the end $E$.
 This is proved in Section \ref{incomp}. With this in place, we prove in Section \ref{qd} 
 that each $A_i$ is quasiconvex in $\til M$ equipped with the graph metric $d_G$ (see
 Lemma \ref{dGqgeod0}, Corollary \ref{dGqgeod} and Lemma \ref{dGqgeod1} in particular).
 
 Next, while in \cite{miyachi-ct, souto-ct}, it is clear for trivial topological reasons that each $A_i$ separates $\til M$, this is no longer true
  in our case. We do have quasiconvexity of $A_i$ in the $d_G-$metric however. It is also true that $A_i \cap \til{\Sigma_i}$ separates
  the lift $\til{\Sigma_i}$ of the split surface $\Sigma_i$. In order to use this weaker separation property, we need to ensure that
if a geodesic segment {\em in $\til H$} joins a pair of points $a,b$, then the geodesic
{\em in $\til M$} joining $a,b$ has approximants built up of vertical pieces in blocks $\til{B_i}$ 
and horizontal pieces (lying on $\til{\Sigma_i}$). We call such approximants {\it admissible quasigeodesics} and deal with their
construction in Section \ref{sec-adm}. 

Once all these ingredients are in place, we prove the  existence of Cannon-Thurston maps for  free degenerate Kleinian groups without parabolics
in Theorem \ref{ct-free}.

Section \ref{ptpreim} then generalizes the work of \cite{mahan-elct} to show that point pre-images of multiple points are given by ending laminations.
We end by mentioning some applications, especially work of Jeon, Kim, Lecuire and Ohshika \cite{woojin} on primitive stable representations
and that of the author \cite{mahan-commens} on discreteness of commensurators of Kleinian groups. In Appendix \ref{app}, we sketch a proof
of the existence of a Minsky model when $M$ has no parabolics.

\section{Preliminaries}\label{prel}
\subsection{Relative Hyperbolicity}
We refer the reader to  \cite{farb-relhyp}
for terminology and details on relative
hyperbolicity and electric geometry. 

\begin{defn} {\rm
Given a metric space $(X,d_X)$ and a collection $\mathcal{H}$ of subsets, let  
$\EE(X,\HH ) = X \bigsqcup_{H \in \HH} (H \times [0,\frac{1}{2}])$ be the 
identification space obtained by  identifying $(h,0) \in H \times [0,\frac{1}{2}]$ 
  with $h \in X$. Each $\{ h \} \times [0,\frac{1}{2}]$ is declared to be
isometric to the  interval $[0, \frac{1}{2} ]$
 and  $H \times \{ \frac{1}{2} \}$ is equipped  with the zero metric. $\EE(X,\HH )$ is given a path pseudo-metric as follows. 

Only such paths in $\EE(X,\HH )$ are allowed whose intersection with any  $\{ h \} \times (0,\frac{1}{2})$ is either all of   $\{ h \} \times (0,\frac{1}{2})$ or is empty.
The distance between two points
in $\EE(X,\HH )$ is the infimum of lengths of such allowable paths.

 The resulting pseudo-metric space 
$\EE(X,\HH )$ is the {\bf electric space} associated to $X$ and the collection $\HH$. \\
We shall say that $\EE(X,\HH )$  is constructed from $X$ by {\bf electrocuting} the collection $\HH$ and the induced pseudo-metric $d_e$ will 
be called the {\bf
electric metric}.\\  (Quasi) geodesics in the electric metric will be referred to as electric  (quasi) geodesics.\\
If  $\EE(X,\HH )$ is (Gromov) hyperbolic, we say that $X$ is {\bf weakly hyperbolic} relative to $\HH$. }
\end{defn}

Note that since  $\EE(X,\HH ) = X \bigsqcup_{H \in \HH} (H \times [0,\frac{1}{2}])$, $X$ can be naturally identified with a subspace of $\EE(X,\HH )$.
Paths in $(X,d_X)$  can therefore be regarded as paths in $\EE(X,\HH )$, but are very far from being quasi-isometrically embedded in general.

A collection $\mathcal{H}$ of subsets of $(X,d_X)$ is said to be $D$-separated if $d_X(H_1, H_2) \geq D$
for all $H_1, H_2 \in \HH; H_1 \neq H_2$.
$D$-separatedness is only a technical restriction as the collection $\{ H \times \{ \frac{1}{2} \}: H \in \HH \}$ is $1$-separated in $\EE(X,\HH )$.

\begin{defn} \label{bt} {\rm  
Given a collection $\mathcal{H}$
of $C$-quasiconvex, $D$-separated sets in a (Gromov) hyperbolic metric space $(X,d_X)$ 
 we
shall say that a geodesic (resp. quasigeodesic) $\gamma$ is a geodesic
(resp. quasigeodesic) {\bf without backtracking} 
 if $\gamma$ does not return to $H$ after leaving it, for any $H \in \mathcal{H}$. \\
There is a distinguished collection of $1$-separated subsets of  $\EE(X,\HH )$ given by $\{ H \times \{ \frac{1}{2} \}: H \in \HH \}$.
An electric  quasigeodesic
 {\bf without backtracking}  in $\EE(X,\HH )$ is an electric  quasigeodesic that does not return to $H  \times \{ \frac{1}{2} \} $ after leaving it, for any $H \in \mathcal{H}$.
}
\end{defn}

\smallskip

\noindent  {\bf Notation:} For any pseudo metric space $(Z, \rho)$ and $A\subset Z$, 
we shall use the notation $N_R(A, \rho) = \{ x \in Z: \rho(x, A) \leq R \}$ as for metric spaces.

\begin{lemma} (Lemma 4.5 and Proposition 4.6 of \cite{farb-relhyp}; Theorem 5.3 of \cite{klarreich}; \cite{bowditch-relhyp})\\
Given $\delta , C$ there exists $\Delta$ such that
if $(X,d_X)$ is a $\delta$-hyperbolic metric space with a collection
$\mathcal{H}$ of $C$-quasiconvex  sets.
Then,\\
 1) {\rm Electric quasi-geodesics electrically track (Gromov) hyperbolic
  geodesics}, i.e. for all $P > 0$, there exists $K > 0$ such that if $\beta$ is any electric $P$-quasigeodesic from $x$ to
  $y$, and  $\gamma$ is a geodesic in $(X,d_X)$ from $x$ to $y$, 
then $\beta \subset N_K ( \gamma, d_e )$. \\
2) $\gamma \subset N_K ((N_0 ( \beta, d_e)), d_X)$. \\
3) {\rm Relative Hyperbolicity:} 
  $X$  is weakly hyperbolic relative to $\HH$. $\EE(X, \HH )$ is $\Delta$-hyperbolic.\\
\label{farb1A}
\end{lemma}

Note that we do not need $D$-separatedness in the hypothesis of Lemma \ref{farb1A} as the definition of electrocution takes care of this..

Let $(X,d_X)$ be a $\delta$-hyperbolic metric
space, and $\mathcal{H}$ a family of $C$-quasiconvex, 
 collection of subsets. 
Let $\alpha = [a,b]$ be a  geodesic in $(X,d_X)$ and $\beta $ 
an electric 
$P$-quasigeodesic without backtracking in $\EE(X, \HH)$
 joining $a, b$. Order from the left the collection of maximal subsegments of $\beta$
contained entirely in some $H  \times \{ \frac{1}{2} \}: H \in \mathcal{H}$. Let $\{ [p_i, q_i]  \times \{ \frac{1}{2} \} ( \subset  H_i
\times \{ \frac{1}{2} \}) \}_i$ be  the collection of maximal subsegments.
Replace, as per this order,  each path of the form $\{p_i\} \times [0,\frac{1}{2}] \cup  [p_i, q_i]  \times \{ \frac{1}{2} \} \cup  \{q_i\} \times [0,\frac{1}{2}] \subset  H_i \times [0,\frac{1}{2}]$ 
 by a   geodesic $[p_i,q_i]$ in $X$. The resulting
{\bf connected}
path $\beta_q$ in $X$ is called an {\em electro-ambient representative} of $\beta$ in
$X$, or simply as the electro-ambient quasigeodesic joining the end-points of $\beta$.

\begin{lemma} (See Proposition 4.3 of \cite{klarreich},  Lemma
  2.5 of \cite{mahan-split}) 
Given $\delta$, $C,  P$ there exists $C_3$ such that the following
holds: \\
Let $(X,d_X)$ be a $\delta$-hyperbolic metric space and $\mathcal{H}$ a
family of $C$-quasiconvex
subsets. Let $(X,d_e)$ denote the electric space obtained by
electrocuting elements of $\mathcal{H}$.  Then, if $\alpha , \beta_q$
denote respectively a (Gromov) hyperbolic geodesic and an electro-ambient
$P$-quasigeodesic with the same end-points in $X$, then $\alpha$ lies in a
(Gromov hyperbolic $d_X-$) 
$C_3$ neighborhood of $\beta_q$.
\label{ea-strong}
\end{lemma}

Two paths   $\beta , \gamma$ in $(X,d_X)$ with  the same endpoints  are said to have \emph{similar intersection patterns} 
with $\HH$ if  there exists $\epsilon >0$, depending only on $(X,\HH)$, such that:
\begin{itemize}
\item {\bf Similar Intersection Patterns 1:}  If
  precisely one of $\{ \beta , \gamma \}$ meets 
   some $H \in \mathcal{H}$, then the $d_X$-distance  from the entry point
  to the 
  exit point is at most $D$. 
\item {\bf Similar Intersection Patterns 2:}  If
 both $\{ \beta , \gamma \}$ meet some  $H \in \mathcal{H}$,
 then the distance  from the entry point of
 $\beta$ to that of $\gamma$ is at most $D$, and  similarly for the exit points. 
\end{itemize}

\begin{definition}  \cite{farb-relhyp} {\rm
Suppose that $X$ is 
 weakly hyperbolic relative to $\mathcal{H}$.
Suppose that any two electric quasigeodesics without backtracking and with the same endpoints  have similar intersection patterns with respect to
the collection $\{ H \times \frac{1}{2} : H \in \mathcal{H} \}$.
Then $(X,\HH)$ is said to satisfy  {\bf bounded penetration}  and 
 $X$ is said to be
{\bf strongly hyperbolic} relative to  $\mathcal{H}$.}
\end{definition}

The next condition ensures that $(X,\HH)$ is  {\bf strongly hyperbolic} relative to  $\mathcal{H}$.

\begin{definition} {\rm A collection $\mathcal{H}$ of uniformly
$C$-quasiconvex sets in a $\delta$-hyperbolic metric space $X$
is said to be {\bf mutually D-cobounded} if 
 for all $H_i, H_j \in \mathcal{H}$, $\pi_i
(H_j)$ has diameter less than $D$, where $\pi_i$ denotes a nearest
point projection of $X$ onto $H_i$. A collection is {\bf mutually
  cobounded} if it is mutually D-cobounded for some $D$. }
  \end{definition}

\begin{lemma} \cite[Proposition 4.6]{farb-relhyp}, \cite{bowditch-relhyp} Given $C, \delta \geq 0$, there exists $P$ such that the following holds:\\
Let $X$ be a $\delta$-hyperbolic metric space and $\HH$ a collection of  $\epsilon$ neighborhoods of mutually cobounded $C$-quasiconvex sets;
then any electro-ambient quasigeodesic is a $(P,P)$ quasigeodesic in $X$.
\label{ea-genl}
\end{lemma}

\noindent {\bf Partial Electrocution} \\
Let $M$ be a  (not necessarily simply connected) convex hyperbolic 3-manifold
with a neighborhood of the cusps excised. Then we can ensure that
each boundary component of $M$ is isometric to $\sigma \times P$, where $P$ is either
an interval or a circle, and $\sigma$ is a horocycle of some fixed 
length $e_0$. In the universal cover $\til M$, if we excise (open) horoballs, we
are left with a manifold whose boundaries are flat horospheres of the
form $\widetilde{\sigma} \times \tilde{P}$. Note that $\tilde{P} = P$ if
$P$ is an interval, and $\mathbb{R}$ if $P$ is a circle (the case for
a $(Z + Z)$-cusp ).

Let $Y$ be a convex simply connected hyperbolic 3-manifold.
Let $\mathcal{B}$ denote a collection of horoballs. Let $X$ denote
$Y$ minus the interior of the horoballs in $\mathcal{B}$. Let 
$\mathcal{H}$ denote the collection of boundary horospheres.  Then each
$H \in \mathcal{H}$ with the induced metric is isometric to a Euclidean
product $E^{1} \times L$ for an interval $L\subset \mathbb{R}$. Here $E^1$ denotes Euclidean $1$-space.
{\bf Partially electrocute}  each 
$H$ by giving it the product of the zero metric with the Euclidean metric,
i.e. on $E^{1}$ put the zero metric and on $L$ put the Euclidean
metric. The resulting space is essentially what one would get (in the spirit of \cite{farb-relhyp}) by gluing
to each $H$ the mapping cylinder of the projection of $H$ onto the $L$-factor. Let $d_{pel}$ denote the partially electrocuted pseudometric
on $X$.

The above construction can be done in the base manifold $M$ itself by equipping the boundary component $\sigma \times P$ with 
the product of a zero metric in the $\sigma$ direction and the Euclidean metric in the $P$-direction.

\begin{lemma} \cite[Lemma 1.20]{mj-pal}
$(X,d_{pel})$ is a (Gromov) hyperbolic metric space.
\label{pel}
\end{lemma}

\subsection{Cannon-Thurston Maps}

Let $(X,{d_X})$ and $(Y,{d_Y})$ be  hyperbolic metric spaces.
By
adjoining the Gromov boundaries $\partial{X}$ and $\partial{Y}$
 to $X$ and $Y$, one obtains their compactifications
$\widehat{X}$ and $\widehat{Y}$ respectively.

Let $ i :Y \rightarrow X$ denote a proper map.

\begin{definition}   Let $X$ and $Y$ be hyperbolic metric spaces and
$i : Y \rightarrow X$ be a proper map.
 A {\bf Cannon-Thurston map} $\hat{i}$  from $\widehat{Y}$ to
 $\widehat{X}$ is a continuous extension of $i$.
\end{definition}

Lemma 2.1 of \cite{mitra-trees} below gives a necessary and sufficient condition for the existence of Cannon-Thurston maps.

\begin{lemma} \cite{mitra-trees}
A Cannon-Thurston map from $\widehat{Y}$ to $\widehat{X}$
 exists iff  the following condition is satisfied:\\
Given ${y_0}\in{Y}$, there exists a non-negative function  $f(n)$, such that 
 $f(n)\rightarrow\infty$ as $n\rightarrow\infty$ and for all geodesic segments
 $\lambda$  lying outside an $n$-ball
around ${y_0}\in{Y}$  any geodesic segment in $X$ joining
the end-points of $i(\lambda)$ lies outside the $f(n)$-ball around 
$i({y_0})\in{X}$.
\label{contlemma}
\end{lemma}

We shall now give a criterion for the existence of Cannon-Thurston maps between relatively hyperbolic spaces.
Let $X$ and $Y$ be strongly   hyperbolic relative to the collections $\HH_X$ and $\HH_Y$ respectively. Let $i\colon Y\to X$ be a weakly type-preserving proper embedding, i.e. for $H_Y\in \HH_Y$ there exists $H_X\in \HH_X$ such that $i(H_Y)\subset H_X$ and images of distinct elements of  $\HH_Y$   lie in distinct elements of $\HH_X$.

In the Lemma below, we specialize to the case where $X, Y$ are convex simply connected complete hyperbolic manifolds with some disjoint (open) horoballs removed. $\HH_X$ and $\HH_Y$
will denote the resulting collections of horospheres.

\begin{lemma}\label{crit-relhyp} \cite[Lemma 1.28]{mj-pal}
A Cannon-Thurston map for a weakly type-preserving proper embedding $i\colon Y \to X$ exists if and only if
there exists a non-negative function $f(n)$ with
$f(n)\rightarrow \infty$ as $n\rightarrow \infty$ such that the
following holds: \\
Suppose $y_0\in Y$,  and $\hat \lambda$ in $\widehat{Y} = \EE(Y, \HH_Y)$ is
an electric quasigeodesic segment  starting and ending  outside horospheres.
If  $\lambda^b = \hat \lambda \setminus \bigcup_{K \in \HH_Y} K$
 lies outside  $B_n (y_0) \subset Y$,
then for any electric quasigeodesic $\hat \beta$ joining the
end points of $\hat i (\hat \lambda)$ in $\widehat{X}= \EE(X, \HH_X)$,
$\beta ^b = \hat \beta \setminus \bigcup_{H \in \HH_X} H$ lies outside
 $B_{f(n)} (i(y_0)) \subset X$.

\end{lemma}

We shall describe this informally as follows: \\
{\it If $\lambda$ lies outside a large ball {\bf modulo horoballs} in $Y$ then so does any geodesic in $X$ joining its endpoints.}

In \cite{mahan-red} we proved the existence of 
Cannon-Thurston maps for Kleinian groups corresponding to
 pared manifolds whose  boundary is incompressible away from cusps.

\begin{definition} A {\bf pared manifold} is a pair $(M,P)$, where $P$, contained in the boundary
$ \partial M$ of $M$, 
is a (possibly empty) 2-dimensional submanifold with boundary such that \\
\begin{enumerate}
\item the fundamental group of each component of $P$ injects into the
fundamental group of $M$
\item  $P$ is a union of annuli and tori.
\item any cylinder $C: (S^1 \times I, \partial ( S^1 \times I)) \rightarrow (M,P)$
such that $\pi_1 (C)$ is injective is homotopic {\it rel} boundary into $P$.
\item $P$ contains every torus  component of $\partial M$.
\end{enumerate}
\end{definition}

A hyperbolic structure adapted to $(M,P)$ is a hyperbolic structure on $M$ such that the parabolics are precisely
the elements of $P$. The collection of such structures is denoted as $H(M,P)$.

 A pared manifold $(M,P)$ is said to have {\bf
  incompressible boundary} 
if each component of $\partial_0 M = \partial M \setminus P$ is
incompressible in $M$.

The following  Theorem summarizes the main results of \cite{mahan-split, mahan-elct, mahan-red}. It proves the existence of 
Cannon-Thurston maps for Kleinian groups corresponding to
 pared manifolds whose  boundary is incompressible away from cusps. It also describes the structure of these maps in terms of ending laminations.

\begin{theorem} \cite{mahan-red}
Suppose that $N^h \in H(M,P)$ is a hyperbolic structure 
on a pared manifold $(M,P)$ with incompressible boundary. Let
$M_{gf}$ denote a geometrically finite hyperbolic structure adapted
to $(M,P)$. Then the map  $i: \widetilde{M_{gf}}
\rightarrow \widetilde{N^h}$ extends continuously to the boundary
$\partial {i}: \partial \widetilde{M_{gf}}
\rightarrow \partial \widetilde{N^h}$. 

Let $E$ be a degenerate end of $N^h$ and $\til E$ a lift of $E$ to $\til{N^h}$. Then the ending lamination $\LL_E$ for the end $E$ lifts to a lamination
on $\til{M_{gf}} \cap \til{E}$. Each such lift $\LL$ of the ending lamination of a degenerate end defines a relation $\RR_\LL$ on the (Gromov) hyperbolic boundary $ \partial \widetilde{M_{gf}}$ given by
$a\RR_\LL b$ iff $a, b$ are  end-points of a leaf of $\LL$. Let $\{ \RR_i \}_i$ be the entire collection of relations on $ \partial \widetilde{M_{gf}}$ obtained this way. Let $\RR$ be the transitive closure of
the union $\bigcup_i \RR_i$. Then $\partial i(a) = \partial i(b)$ iff $a\RR b$.
\label{ptpre-inde}
\end{theorem}

\section{Split Geometry} \label{min}

We recapitulate the essential aspects of split geometry from \cite{minsky-elc1, mahan-split}.

\smallskip

\noindent {\bf Split level Surfaces}\\
A   {\bf pants decomposition} of a compact surface $S$, possibly with boundary,
is a disjoint collection of
3-holed spheres $P_1, \cdots , P_n$ embedded in $S$ such that $S \setminus \bigcup_i P_i$ is a disjoint collection of non-peripheral
annuli in $S$, no two of which are  homotopic. 

Let $N$ be the convex core of a  hyperbolic 3-manifold minus an open neighborhood of the cusp(s). Then any end $E$ of $N$ is simply degenerate
 \cite{agol-tameness, gab-cal, canary} and
 homeomorphic to $S \times [0, \infty )$, where $S$ is a compact surface, possibly with boundary. A closed geodesic in an end $E$ 
 homeomorphic to $S \times [0, \infty )$  is {\bf unknotted} if it is isotopic in $E$ to a simple closed curve in $S \times \{0 \}$ via the homeomorphism.
A {\bf tube} in an end $E \subset N$ is a regular $R-$neighborhood
$N(\gamma, R)$ of an unknotted geodesic $\gamma$ in $E$.

Let $\TT$ denote a collection of disjoint, uniformly separated tubes in ends of  $N$ 
such that

\begin{enumerate}
\item[a)] Otal \cite{otal-unknot} showed the existence of an $\epsilon_{ot} >0$ (depending only on the topology of the end $E$)
such that any primitive closed geodesic of length at most $\epsilon_{ot} >0$ (referred to below as the Margulis-Otal constant)
is unknotted.  We shall refer to tubes around such geodesics as
 Margulis tubes (strictly speaking, $\epsilon_{ot}$ is smaller than the Margulis constant; so this is a slight abuse of terminology).
 All Margulis tubes in $E$ belong to $\TT$ for all ends $E$ of $N$.
\item[b)] there exists $\epsilon_0 >0$ such that the injectivity radius $injrad_x(E) > \epsilon_0$ for all $x \in E \setminus  \bigcup_{T \in \TT} Int(T)$ and all ends $E$ of $N$.
\end{enumerate}

In \cite{minsky-elc1}, Minsky constructs a
model manifold $M$ biLipschitz homeomorphic to $N$ and equipped with a piecewise Riemannian structure. We shall refer to this
model as the {\bf Minsky model}.
The features we shall use for $M$ will be given below.
For the time being, we note that $M$ has a collection of tubes, each with a Euclidean structure on its boundary.
The complement of these tubes may be decomposed as a union of blocks of some standard types. A Lipschitz map $F$ was constructed in 
\cite{minsky-elc1} from $N$ to $M$ and it was shown by Brock-Canary-Minsky \cite{minsky-elc2} (see \cite{bowditch-endinv} for the general Kleinian groups
case) that $F$ was, in fact bi-Lipschitz.
\footnote{
The bi-Lipschitz homeomorphism between the model manifold $M$ and the hyperbolic manifold $N$
is established by Brock-Canary-Minsky in \cite{minsky-elc2} for 
(pared) manifolds with incompressible boundary. The extension to the general case was sketched briefly in 
\cite{minsky-elc2}, but details have not appeared in print. The paper \cite{bowditch-endinv} is also
in preprint form as of date. To get around this, we provide a brief sketch in the Appendix
\ref{app} using published work of several authors, reducing the general case to the case with incompressible boundary.}
Thus, let $F: N \rightarrow M$ be a  bi-Lipschitz homeomorphism to the model manifold $M$ and let $M(0)$
be the image of $N \setminus \bigcup_{T \in \TT} Int(T)$ in $M$ under   $F$.
 Let $\partial M(0)$ (resp. $\partial M$) denote the  boundary  of $M(0)$ (resp. $M$). The metrics on $M$ and $\til M$ will be denoted by $d_M$.

Let $(Q, \partial Q)$ be the unique hyperbolic  pair of pants such that each  component
of $\partial Q$ has length one. $Q$ will be called
the {\it standard} pair of pants.
An isometrically embedded copy of $(Q, \partial Q)$ in $(M(0), \partial M(0))$ will be said to be {\it flat}.

\begin{defn} {\rm A {\bf split level surface} associated to a pants decomposition $\{ Q_1, \cdots , Q_n \}$ of 
a compact surface $S$ (possibly with boundary) in $M(0) \subset M$
is an embedding $f : \cup_i (Q_i, \partial Q_i) \rightarrow (M(0), \partial M(0))$ such that \\
1) Each $f (Q_i, \partial Q_i)$ is flat \\
2) $f$ extends to an embedding (also denoted $f$) of $S$ into $M$ such that the interior of each annulus component of
$f(S \setminus \bigcup_i Q_i)$  lies entirely in $F(\bigcup_{T \in \TT} Int(T))$. \\
} \end{defn}

Let
$S_{i}^{s}$ denote the union of the collection of flat pairs of pants
 in the image of the embedding  $S_{i}$.  Note that $S_i \setminus S_{i}^{s}$ consists of annuli properly embedded in Margulis tubes.

The class   of {\it all} topological embeddings from $S$ to $M$ that agree with a split level surface $f$ 
associated to a pants decomposition $\{ Q_1, \cdots , Q_n \}$ on 
$Q_1 \cup \cdots \cup Q_n$ will be denoted by $[f]$. 

We define a partial order $\leq_E$ on the collection of split level surfaces in an end $E$ of $M$ as follows: \\
$f_1 \leq_E f_2$ if there exist $g_i \in [f_i]$, $i=1,2$, such that $g_2(S)$ lies in the unbounded component of $E \setminus g_1(S)$.

A sequence $S_i$ of split level surfaces is said to exit an end $E$ if $i<j$ implies $S_i \leq_E S_j$ and further for all compact subsets $B \subset E$, there exists
$L>0$ such that $S_i \cap B = \emptyset$ for all $i \geq L$.

\begin{definition}\label{thin}
  A curve $v$ in $S \subset E$ is {\bf $l$-thin} if the core curve of the Margulis tube $T_v (\subset E \subset N)$ has length less than or equal to $l$. 
A   tube $T\in \TT$  is   $l$-thin  if its core curve    is   $l$-thin.  A   tube $T\in \TT$  is   $l$-thick if it is not    $l$-thin.  \\
A curve $v$ is said to split a pair of split level surfaces $S_i$ and $S_j$ ($i<j$) if $v$ occurs as a boundary curve of
 both $S_i$ and $S_{j}$. A pair of split level surfaces $S_i$ and $S_j$ ($i<j$) is said to be an {\bf $l$-thin pair} if there exists an   $l$-thin curve $v$ 
 splitting both  $S_i$ and $S_{j}$.  \\

\end{definition}

The collection of all  $l$-thin tubes is denoted as $\TT_l$. The union of all  $l$-thick tubes along with $M(0)$  is denoted as $M(l)$.
Unless otherwise indicated, we shall choose $l$ to be the Margulis-Otal constant, though the discussion below will go through for any
$l<\epsilon_{ot}$.

\begin{defn}
A pair of split level surfaces $S_i$ and $S_j$ ($i<j$) is said to be {\bf $k$-separated} if \\
a) for all $x \in S_i^s$, 
$d_M(x,S_j^s) \geq k$\\
b)  Similarly, for all $x \in S_j^s$, $d_M(x,S_i^s) \geq k$. \end{defn}

\begin{defn} {\rm An $L$-bi-Lipschitz {\bf split  surface} in $M(l)$ associated to a pants decomposition $\{ Q_1, \cdots , Q_n \}$ of $S$
and a collection $\{ A_1, \cdots , A_m \}$ of complementary annuli (not necessarily all of them) in $S$ 
is an embedding $f : \cup_i Q_i \bigcup  \cup_i A_i \rightarrow M(l)$ such that\\
1) the restriction  $f: \cup_i (Q_i, \partial Q_i) \rightarrow (M(0), \partial M(0))$ is a split level surface \\
2) the restriction $f: A_i \rightarrow M(l)$ is an $L$-bi-Lipschitz embedding.\\
3)  $f$ extends to an embedding (also denoted $f$) of $S$ into $M$ such that the interior of each annulus component of
$f(S \setminus (\cup_i Q_i \bigcup  \cup_i A_i))$  lies entirely in $F(\bigcup_{T \in \TT_l} Int(T))$.}\end{defn}

\noindent {\bf Note:} The difference between a split level surface and a split surface is that the latter may contain
bi-Lipschitz annuli in addition to flat pairs of pants.

\smallskip

We denote split surfaces by  $\Sigma_{i}$ to distinguish them from split level surfaces $S_i$.
Let
$\Sigma_{i}^{s}$ denote the union of the collection of flat pairs of pants
and bi-Lipschitz annuli in the image of the split surface (embedding)  $\Sigma_{i}$. The next Theorem is one of the technical tools from
\cite{mahan-split}. For the convenience of the reader, we 
recall  the following representative situation from the Introduction to \cite{mahan-split}.

\begin{enumerate}
\item there exists a sequence
$\{ S_i \}$  of   disjoint,  embedded, bounded geometry surfaces exiting  $E$. These are ordered in a natural way along $E$, i.e. $i<j$ implies that
$S_j$ is contained in the unbounded component of $E \setminus S_i$. The topological product region  between $S_i$ and $S_{i+1}$ is denoted $B_i$.
\item Each product region $B_i$ is of two types: \\
Either the product region $B_i$ is {\bf thick}, i.e. there exists a uniform (independent of $i$) constant $K' \geq 1$
such that $B_i$ is $K'-$biLipschitz homeomorphic to $S_i \times [0,1]$. Such product regions are called {\bf thick blocks};\ \\
Or, corresponding to the product region $B_i$, there exists a Margulis tube $T_i$ such that $T_i \subset B_i$. Further, 
 $T_i \cap S_i$ and  $T_i \cap S_{i+1}$ are annuli on $S_i$ and $S_{i+1}$ respectively, with core curves homotopic to the  core curve of $T_i$. These are examples of {\bf split blocks}.
\end{enumerate}

Thus, for split blocks,
the $T_i$ {\it split} both $S_i$ and $S_{i+1}$. The complementary pieces (and their lifts to the universal cover) are examples of
{\it split components}. Note that we have little control, however, on the geometry of the split components. This situation generalizes 
to give a sequence of split  surfaces (rather than actual surfaces as in the representative situation above)
exiting the end, such that successive pairs are split by some Margulis tubes:

\begin{theorem} \cite[Theorem 4.8]{mahan-split}
Let $N, M, M(0), S, F$ be as above and $E$ an end of $M$. Fix $l$ less than the Margulis-Otal constant, and
let $M(l) = \{ F(x) : {\rm injrad_x} (N) \geq l \}$. Fix a hyperbolic metric on $S$ such that each component of $\partial S$ is 
totally geodesic of length one (this is a normalization condition).
 There exist $ L_1 \geq 1$, $  \epsilon_1 > 0$, $n \in \natls$, 
 and a sequence $\Sigma_i$ of $L_1$-bi-Lipschitz, $  \epsilon_1$-separated split  surfaces exiting the end $E$ of $M$
such that for all $i$, one of the following occurs: \\
\begin{enumerate}
\item An $l$-thin curve $v$ splits the pair $(\Sigma_i ,\Sigma_{i+1})$, i.e. $v$ splits the associated split level surfaces $(S_i ,S_{i+1})$, which in turn  form
an $l$-thin pair. 
\item there exists an $L_1$-bi-Lipschitz embedding  $$G_i: (S\times [0,1], (\partial S)\times [0,1]) \rightarrow (M, \partial M),$$ (equipping
$S\times [0,1]$ with the product metric)
such that $\Sigma_i^s = G_i (S\times \{ 0\})$ and $\Sigma_{i+1}^s = G_i (S\times \{ 1\})$
\end{enumerate}
Finally, each $l$-thin curve in $S$ splits at most
$n$ split level surfaces in the  sequence $\{ \Sigma_{i} \}$. \label{wsplit}
\end{theorem}

In Theorem \ref{wsplit} above, $n$ depends on the genus of the surface $S$. The exact nature of this dependence is
in terms of hierarchies and is explicated in \cite{mahan-split}; it is not important for this paper.
  
A model manifold $M$ all of whose ends are equipped with a collection of  exiting    split  surfaces satisfying the conclusions of Theorem \ref{wsplit} is said to be equipped
with a {\bf weak split geometry} structure.

As mentioned in Definition \ref{thin}, 
pairs of split surfaces satisfying Alternative (1) of Theorem \ref{wsplit} will be called an $l$-thin pair of split surfaces (or simply a thin
pair if $l$ is understood). Similarly, pairs of split surfaces satisfying Alternative (2) of Theorem \ref{wsplit} will be called an $l$-thick pair
(or simply a thick
pair) of split surfaces.

\begin{defn} Let $(\Sigma_i^s, \Sigma_{i+1}^s)$ be a thick pair of split surfaces in $ M$. 
The closure of the bounded component of
$M \setminus (\Sigma_i^s \cup \Sigma_{i+1}^s)$ between  $\Sigma_i^s, \Sigma_{i+1}^s$   will be called a thick block.\end{defn}

Note that a thick block is uniformly bi-Lipschitz to the product $S \times [0,1]$ and that its boundary components are 
$\Sigma_i^s, \Sigma_{i+1}^s$.

\begin{defn} Let $(\Sigma_i^s, \Sigma_{i+1}^s)$ be an $l$-thin pair of split surfaces in $M$
and $F(\TT_i)$ be the collection of $l$-thin Margulis tubes that split both $\Sigma_i^s, \Sigma_{i+1}^s$. The closure of the union of the
bounded components of
$M \setminus ((\Sigma_i^s \cup \Sigma_{i+1}^s)\bigcup_{F(T)\in F(\TT_i)} F(T))$  between  $\Sigma_i^s, \Sigma_{i+1}^s$  will be called a {\bf split block}.
The closure of any bounded component is called a {\bf split component}. 
\end{defn}

Note that each split component may contain Margulis tubes that {\em do not} split both $\Sigma_i^s, \Sigma_{i+1}^s$.

\begin{rmk}  {\rm   For each lift $\til{K} \subset
\til{M}$ of a split component $K$ of a split block of $M(l) \subset M$, 
there are lifts of $l$-thin Margulis tubes that share the boundary of  $\til{K}$ in $\til{M}$. Adjoining these lifts to 
$\til{K}$ we obtain {\bf extended split components}. Let $\KK^\prime$ denote the collection of extended split components  in $\til{M}$.
Denote  the collection of split components in $\til{M(l)} \subset \til{M}$ by $\KK$.
Let $\til{M(l)}$ denote the lift of $M(l)$ 
to $\til M$. 
Then the inclusion of $\til{M(l)}$ into $\til{M}$ gives a quasi-isometry between $\EE (\til{M(l)}, \KK)$ and
$\EE (\til{M}, \KK^\prime)$ equipped with the respective electric metrics. This  follows from the last assertion of Theorem \ref{wsplit}.

Note here that two split components may intersect along a flat subsurface along a common horizontal boundary component, and that two 
extended split components may intersect along Margulis tubes in addition. However, for the electrocution operation, this does not pose any problems.
This is because while electrocuting,
products with the unit interval of the form $\til{K} \times [0,\frac{1}{2}]$ are attached to
 $\til{M}$ by identifying $\til{K} \times \{ 0 \}$ with $\til{K} (\subset \til{M})$
and then each element of the collection $\til{K} \times \{ \frac{1}{2} \}$ is given the zero metric.

 The electric metric on $\EE (\til{M}, \KK^\prime)$ is called the {\bf graph-metric} and
is denoted by $d_G$. The electric space will be denoted as $(\til{M}, d_G)$. If there are no thick blocks, the graph
metric between two points $x, y$ roughly measures the minimal number of split components one has to pass through
to go from $x$ to $y$.

The electric metric on $\EE (\til{M}, \KK \bigcup \TT_l)$ is quasi-isometric to the electric metric on $\EE (\til{M}, \KK^\prime)$, again by  the last assertion of Theorem \ref{wsplit}. 
The electric space will be denoted as $(\til{M}, d_G^1)$.

Note also that in $\EE (\til{M}, \KK \bigcup \TT_l)$,  elements of $\KK$ and $\TT$ are electrocuted separately. On the other hand each element in 
$\KK^\prime$ is a union of an element of $\KK$ and abutting elements of $\TT$. Thus we cannot say that $\EE (\til{M}, \KK \bigcup \TT_l)$ and 
$\EE (\til{M}, \KK^\prime)$ are isometric.}
\end{rmk}

\medskip

\begin{definition} Let $Y \subset \til{N}$ and $X=F(Y)$.  $X \subset \til{M}$ is said to 
be $\Delta$-graph quasiconvex if for any hyperbolic geodesic $\mu$ joining $a, b \in Y$,
$F(\mu )$ lies inside $N_\Delta (X, d_G) \subset \EE (\til{M}, \KK^\prime)$. \end{definition}

For $X (=F(Y))$ a split component in a manifold, define $CH(X) = F(CH(Y))$, where $CH(Y)$
  is the convex hull of $Y$ in $\til{N}$, provided the ends of $N$ have no cusps, i.e. $N=N^h$. Else define $CH(X)$ to be the image under
$F$ of $CH(Y)$ minus cusps.
Further, in order to ensure hyperbolicity of the universal cover, we partially electrocute the cusps of $M$ (cf. Lemma \ref{pel}).

Then $\Delta$-graph quasiconvexity of $X$
is equivalent to the condition that  $dia_G (CH(X))$  is bounded by $\Delta^{\prime} = \Delta^{\prime}(\Delta )$ as any split component 
has diameter one in $(\til{M}, d_G)$.

A split component $K (\subset E) \subset N$ is incompressible if the map $i_\ast : \pi_1(K) \rightarrow \pi_1(N)$  induced by the inclusion is injective.
Lemma \ref{hypqc-a}, Proposition \ref{gr-qc-free-a} and Proposition \ref{dGhyp-a} below were proved in \cite{mahan-split} for $M$ homotopy equivalent to a surface, where all split
components are automatically incompressible. However the proofs in  \cite{mahan-split} require only that the split
components be incompressible in $M$.

\begin{lemma}\cite[Lemma 4.16]{mahan-split} Let $E$  be a simply  degenerate end of a hyperbolic 3-manifold $N$ equipped with
a weak split geometry model $M$.
For $K$ an incompressible split component contained in $E$, let $\tilde{K}$ be a lift to $\til N$.  Then there exists $C_0
= C_0(K)$ such that  the convex hull of $\tilde K$ minus cusps lies in a
$C_0$-neighborhood of $\tilde K$ in $\til N$. \label{hypqc-a}
\end{lemma}

\begin{prop}\cite[Proposition 4.23]{mahan-split} If $K$ is an incompressible split component, then $\til{K} $ is
  uniformly graph-quasiconvex in  $\til{M}$, i.e. there exists $\Delta^\prime$ such that $dia_G(CH(\til{K}))) \leq  \Delta^\prime$
for all incompressible split components $\til{K} $.
\label{gr-qc-free-a}
\end{prop}

\begin{prop}\cite[Corollary 4.30]{mahan-split} Suppose that all split components of $\til M$ are incompressible. Then $ (\til{M}, d_{G})$ and hence  $ (\til{M}, d_{G}^1)$   are Gromov-hyperbolic.
\label{dGhyp-a}
\end{prop}

In fact electro-ambient quasigeodesics in $ (\til{M}, d_{G})$ and   $ (\til{M}, d_{G}^1)$ have the following relation.

\begin{lemma}
 Let $M$ be  a model of split geometry such that all split components are incompressible. Let  $ (\til{M}, d_{G}) (= \EE(\til{M},\KK^\prime ))$ and   $ (\til{M}, d_{G}^1)
 (= \EE(\til{M},\KK \bigcup \TT_l ))$ be as above.
Given $o \in \til{M}$ and $C_0 > 0 $, 
there exists a function $\Theta : \mathbb{N} \rightarrow \mathbb{N}$ satisfying
$\Theta (n) \rightarrow \infty$ as $n \rightarrow \infty$ such that the following holds.\\
For any $a, b \in \til{M}$, let $\beta_{ea}^h $  be an electro-ambient
$C_0-$quasigeodesic without backtracking in $ (\til{M}, d_{G})$ joining $a, b$. Let
$\beta_{ea} = \beta_{ea}^h \setminus \partial \til{M}$ be the part of $\beta_{ea}^h$ lying away from
the (bi-Lipschitz) horospherical boundary of $ \til{M}$. Again, let $\beta_{ea1}^h $  be  an electro-ambient
$C_0-$quasigeodesic without backtracking in $ (\til{M}, d_{G}^1)$ joining $a, b$. Let
$\beta_{ea1} = \beta_{ea1}^h \setminus \partial \til{M}$ be the part of $\beta_{ea1}^h$ lying away from
the (bi-Lipschitz) horospherical boundary of $ \til{M}$, where we refer to the bi-Lipschitz image under $F$ of the
horospherical boundary of $ \til{N}$ as the (bi-Lipschitz) horospherical boundary of $ \til{M}$. 

Then $d_M(\beta_{ea}, o) \geq n $ implies that  $d_M(\beta_{ea1}, o) \geq \Theta (n)$. Conversely, 
$d_M(\beta_{ea1}, o) \geq n $ implies that  $d_M(\beta_{ea}, o) \geq \Theta (n)$.
\label{contlemma1}
\end{lemma}

\begin{proof} Let $K^\prime$ be an extended split component in $\KK^\prime$ and $\til{K^\prime}$ denote its universal cover.  Let
$\til{K^\prime} = \til{K} \bigcup \til{T^i}$ where $\til{T^i}$ are the universal covers  of $l$-thin Margulis tubes abutting the split component 
$\til{K}$ in $\til{K^\prime}$.   Suppose $K$ is contained in $ B_i$, the $i$-th  block in an end $E$.

Then  $\til{K^\prime}$ is hyperbolic and is contained in a $C(=C(K^\prime))$-neighborhood of $\til K$.  The argument is now a reprise of similar arguments
 in Section 6  (e.g. Lemma 6.8 and 6.10) of \cite{mahan-split}:\\
For all $i$, there exists $C(i)$, such that $\beta_{ea1} \cap \til{B_i}$ lies in a $C(i)$-neighborhood of 
 $\beta_{ea} \cap \til{B_i}$ in $\til{M}$. Suppose
$d_M(\beta_{ea}, o) \geq n $.  Hence, by uniform $k_0$-separatedness of split surfaces (Theorem \ref{wsplit}),
$d_M(\beta_{ea1} \cap \til{B_i}, o) \geq max(n-C(i), ik_0)$.

Let $D(i) =max_{1\leq j \leq i} C(i)$. Then $d_M(\beta_{ea1}, o) \geq max(n-D(i), ik_0)$ for all $i$. One direction of the Lemma follows.

The converse direction is similar.
\end{proof}

We summarize the conclusions of the above propositions below.

\begin{defn} \label{gqc} A model manifold of weak split geometry is said to be of {\bf split geometry} if \\
\begin{enumerate}
\item  Each split component $\til{K}$ is 
quasiconvex (not necessarily uniformly) in the hyperbolic metric on $\til{N}$. \\
\item Equip $\til{M}$ with the {\it graph-metric} $d_G$ obtained by
electrocuting  (extended) split components $\til{K}$. Then the convex hull
$CH( \til{K})$ of any split component $\til{K}$ has uniformly bounded
diameter in the metric $d_G$. 
\end{enumerate}
\end{defn}

Hence by Lemma \ref{hypqc-a} and Proposition \ref{gr-qc-free-a} we have the following (where we refer the reader to
the Appendix, Section \ref{app}, for a sketch of a proof 
of the existence of a
Minsky model 
for a  general finitely generated Kleinian group without parabolics).

\begin{theorem}
Any  degenerate end $E$ of a hyperbolic 3-manifold $M$
 admitting a Minsky model also
has a model of split geometry. In particular, if $j: E\to M$ denotes the inclusion 
and if no element of $j_\ast (\pi_1(E))$ is a parabolic, then $E$ admits a model of split geometry. \label{minsky-split}
\end{theorem}

\section{Free Groups and Finitely Generated Kleinian Groups}\label{mainsec}

Let $G$ be a  free geometrically infinite Kleinian group.
Agol \cite{agol-tameness}, and independently, Gabai and Calegari \cite{gab-cal} have shown that 
$N = {\Hyp^3}/G$ is topologically tame, and hence, by work of Canary 
\cite{canary}, geometrically tame. Then any manifold $M$ bi-Lipschitz to $N$ is homeomorphic to the
interior of a handlebody with boundary $S$. 

More generally, let $G$ be a finitely
generated Kleinian group  and $G_f$ be  a geometrically finite
Kleinian group, abstractly isomorphic to $G$ via a strictly type-preserving
isomorphism. Let $H$ denote the convex core of ${\Hyp}^3 / G_f$ and let $M$ be a model manifold bi-Lipschitz homeomorphic to 
$N = {\Hyp}^3 / G$. Then there is a natural identification $i: H \rightarrow
M$ of $H$ with the augmented Scott core (i.e. Scott core plus parabolics) of $M$. Let $\til{i}$ indicate
the lift of $i$ to the universal cover. For most of the discussion below, it might be helpful at  first reading to have in mind a free
geometrically infinite Kleinian group without parabolics. We fix this notation for $H, M, S$ throughout this section.

\noindent {\bf Standing Assumption:} For the purposes of this Section, we assume that each degenerate end $E$ of $M$ admits
a Minsky model. As mentioned in Remark \ref{minapp} this is expected to be satisfied always and a proof in the special case that
$M$ has no parabolics is sketched in Section \ref{app}.

\subsection{The Masur Domain}

\begin{definition} Let $E$ be an end of a hyperbolic manifold such that $E$ is homeomorphic to $ S \times [0, \infty )$ for $S$ a finite area hyperbolic surface.
A map $h : E \rightarrow S \times [0, \infty )$ is said to
  be type-preserving, if all and only
 the  cusps of $E$ are mapped to cusps of $S \times [0, \infty )$. 
\label{def-typepres}
\end{definition}

\begin{theorem} \cite{agol-tameness, gab-cal, bowditch-endinv, minsky-elc3}
Let $G$ be a finitely generated Kleinian group and $M = {\Hyp^3}/G$. Let
$H$ denote an augmented Scott core of $M$. Let $E_1$ be a
geometrically infinite end of
$M \setminus H$. Then $E_1$ is homeomorphic (via a type-preserving
homeomorphism) to a topological product  $S \times [0, \infty )$ for a
  hyperbolic
surface $S$ of finite area. Further, there exists a neighborhood $E$
of the end corresponding to $E_1$ such that $E$ is bi-Lipschitz
homeomorphic to a Minsky model for $S \times [0, \infty )$ and 
hence to a model of split geometry. 
\label{free-model}
\end{theorem}

The last part of the last statement follows from Theorem \ref{minsky-split}. 

Some ambiguity remains in the statement of Theorem \ref{free-model}
above. This lies in the choice of the ending lamination for $E$ used
to build the Minsky model. Since $i : S \subset E$ is type-preserving,
no parabolic element of $S$ bounds a compressing disk. 
Let $\ML(S)$ be the space of measured laminations on $S$ and let $\PML (S)$ denote the space of projectivized 
measured laminations. Let $\DD(S)$ be the
subset of $\PML(S)$ consisting of weighted unions of disjoint
meridians (boundaries of compression disks lying on $S$). Let
$cl(\DD(S) )$ denote the closure of $\DD(S)$. Define the {\bf Masur domain}
of $S$ by \\
$\PP\MM \DD (S) = \{ \lambda \in \PML(S) : i(\lambda, \mu ) > 0\}$ for all
  $\mu \in cl(\DD (S) )$, provided $S$ has at least two disjoint isotopy classes of compressing disks.

 Else, we define \\
$\PP\MM \DD (S) = \{ \lambda \in \PML(S) : i(\lambda, \mu ) > 0$ 
 for any
  $\mu $ that is disjoint from a compressing disk $\}$.

Here $i(\lambda, \mu )$ denotes the intersection number of $\lambda, \mu$. Let $\MM \DD (S) (\subset \ML (S))$ be 
the set of measured geodesic laminations whose projective
class lies in  $\PP\MM \DD (S)$.

Now, let $M = H \cup_i E_i$, where $H$ is an augmented Scott core.
We consider the ends $E_j$ which have compressible boundary, i.e. $\partial E_j (= E_j \cap H)$
 is compressible in $H$. Let $E$ be such an $E_j$
 and $S = H \cap E$ be the corresponding
 boundary component of $H$. Then any compressible simple closed curve on $\partial H$
 lies on such an $S$. We fix an $S$ for now and proceed. Let $Mod_0(S)$
denote the subgroup of the mapping class group of $S$ generated by
Dehn twists along essential simple closed curves that bound embedded disks in $H$. 
 $Mod_0(S)$ acts on $\PP\MM\LL(S)$. It was shown
by Otal \cite{otal-thesis} (see also McCarthy and Papadopoulos
\cite{mcpapa}) that under this action, $Mod_0 (S)$ acts properly discontinuously on
$\PP\MM\DD(S) (\subset \PP\MM\LL(S))$ with limit set $ cl(\DD(S)) (\subset \PP\MM\LL(S))$. 

\begin{rmk} Classically \cite{luft,suzuki,mcmil}, $Mod_0(S)$ is defined as the subgroup of the mapping class group
$Mod(S)$ of $S$ which extend to the trivial outer automorphism of $\pi_1(H)$. Luft \cite{luft}, Suzuki \cite{suzuki}
and McCullough-Miller \cite{mcmil}
prove that this group is exactly the subgroup of the mapping class group of $S$ generated by
Dehn twists along essential simple closed curves that bound embedded disks in $H$. 
\end{rmk}

As mentioned in the introduction to the paper, for any degenerate end $E$ with boundary surface $S$,
 there exists a sequence
of simple closed curves $\{ \sigma_n \}$ on $S$, whose geodesic realizations exit the end $E$. 
The limit of these curves in $\PML(S)$ defines an ending lamination for $E$. However, since $S$ is compressible (in $H$), 
we may obtain a different sequence of simple closed curves $\{ \sigma_n' \}$ on $S$ by acting on $\sigma_n$ by different elements of $Mod_0(S)$.
Note that $\sigma_n'$ and $\sigma_n$ are homotopic in $H \cup E$ and hence have the same geodesic realization in $E$. Otal \cite{otal-thesis} (see also \cite{ks} for the case with parabolics) shows   
that any two essential simple closed curves on $S$  in the Masur domain 
that are freely homotopic in $M$ lie in the same $Mod_0 (S)-$orbit.
We refer the reader to   \cite[Proposition 3.3]{canary} for a published proof .
Otal \cite{otal-thesis} (see also \cite{ks}) further shows that the subset of weighted multicurves in
$\PP\MM \DD (S)/Mod_0(S)$ injects homeomorphically
into the space of currents on $M$. It follows that the ending lamination for $E$ is well-defined up to the action of
$Mod_0(S)$:

\begin{theorem} 
For any finitely generated Kleinian group, the ending lamination
$\lambda$ 
facing a surface $S$ with a  compressing disk lies in the Masur Domain and is a well-defined element of $\PP\MM \DD (S)/Mod_0(S)$.
\label{elinmd}
\end{theorem}

For our purposes we shall mostly be satisfied with the fact that $E$
is bi-Lipschitz homeomorphic to {\em some} Minsky model, and hence, by Theorem \ref{free-model} to a model
of split geometry.

\subsection{Incompressibility of Split Components}\label{incomp}
Recall that we are working in the setup where $M$ is a hyperbolic 3-manifold and $H$ its Scott core (augmented Scott core, when $M$ has parabolics). Let $S$ be a boundary component of $H$ that is compressible (rel. cusps when $M$, and hence $H$, has parabolics).
 Let $E$ be the end with $S$ as its boundary. 
We would like to show that sufficiently deep within  $E$, all split
components are incompressible {\em in $M$}. 
Recall 
that splitting tubes correspond to {\em thin Margulis tubes} in the
split geometry model built from the Minsky model.

\begin{prop}
Let $M, H, E$ be as above. Equip $E$ with a split geometry structure. Then there exists (a "sub-end") $E_2 \subset E$ such that \\
\begin{enumerate}
\item $E_2$ is homeomorphic to $S \times [0, \infty )$ by a
  type-preserving homeomorphism and consists of a union of blocks and tubes from
  the split geometry model for $E$.
\item All split components of $E_2$ are incompressible, i.e. if $K$ is
  a split component of $E_2$, then the inclusion $i : K \rightarrow M$
  induces an injective map $i_\ast : \pi_1 (K) \rightarrow \pi_1 (M)$.
\end{enumerate}
\label{incompressible}
\end{prop}

\begin{proof} Suppose not. Then there exists a sequence of
split components $K_i$ exiting the end $E$ such that 
 $i_\ast : \pi_1 (K) \rightarrow \pi_1 (M)$ is not injective. Since
$K_i$ are split components, $K_i = S_i \times I$ for some subsurface
$S_i$ of $S$. By the
Loop Theorem (see for instance, Hempel \cite{hempel-book}), there
exist simple closed curves $\sigma_i \subset S_i$ such
that $\sigma_i$ bound embedded topological disks in $H$. Hence
$\sigma_i \in  \DD (S)$. Let $T_i$ be a splitting tube bounding
$K_i$. Then $T_i$ exit $E$ and has a core curve $\alpha_i$ which in turn corresponds to a simple closed curve on $S$. It
follows that $\alpha_i$ is disjoint from an essential disk and hence
some $\sigma_i' \in \DD(S)$. Since $\alpha_i$ are simple closed curves
whose geodesic realizations exit $E$,
it follows that
any such
sequence of curves $\alpha_i$ converges (in $\PML(S)$) to the ending lamination
$\lambda$ corresponding to the end $E$. Hence
$\lambda \in cl(\DD(S))$ and  cannot lie in the Masur
domain. This contradicts Theorem \ref{elinmd}. The proposition follows. \end{proof}

Henceforth we shall choose $l$ (while fixing $l-$thin Margulis tubes in the construction of the split geometry model)
to be small enough, so that  $(E \setminus E_2)$ is contained in a thick block and hence  all split components are incompressible.
As an immediate consequence we have the following.

\begin{lemma}
If $K$ is a split component, then $\pi_1(K)( \subset \pi_1(M))$ is
geometrically finite (Schottky, in the absence of parabolics).
\label{hypqc-free}
\end{lemma}

\begin{proof} Follows from Lemma \ref{hypqc-a}. \end{proof}

\begin{prop} If $K$ is a split component, then $\til{K} $ is
  uniformly graph-quasiconvex in  $\til{M}$.
\label{gr-qc-free}
\end{prop}

\begin{proof} Follows from Proposition \ref{gr-qc-free-a}. \end{proof}

\begin{prop}  
$(\til{M}, d_G)$ is a hyperbolic metric space.
\label{graph_metric_hyperbolic}
\end{prop}

\begin{proof} Follows from Proposition \ref{dGhyp-a}. \end{proof}

\smallskip

\subsection{Constructing Quasidisks}\label{qd}

The construction in this subsection may be regarded as a  {\it
  graph-metrized coarse} 
analogue of an unpublished  construction due to Miyachi \cite{miyachi-ct} (see also
  Souto \cite{souto-ct} ). The main technical difference between Miyachi's construction and ours is that Miyachi
  constructs  continuous images of disks that actually separate the universal cover $\til M$, 
whereas we only construct quasidisks. As a consequence it becomes technically
  more difficult for us to prove that  quasidisks coarsely separate. This is why we need a special family of paths which we shall call 'admissible
  paths' in the next subsection which either intersect or come close to the quasidisks we construct below.

Recall that
$M$ is a hyperbolic 3-manifold and $H$ its Scott core (augmented Scott core, when $M$ has parabolics). Also, $E$
is an end such that  $S = H\cap E$ is compressible. 
We choose a
collection of essential simple closed curves $\sigma_1 \cdots \sigma_g$ on $S$ bounding
disks $D_1 \cdots D_g$ with neighborhoods $D_i \times ( -\epsilon ,
\epsilon )$ such that each component of $H \setminus
\bigcup_i D_i \times ( -\epsilon ,
\epsilon )$ is either a ball or has incompressible boundary (rel. cusps).  Also assume that $\sigma_i$ are geodesics in
the intrinsic metric on $S$. To avoid multiple indices we fix an end $E$ of $M$ and
describe the construction of quasidisks in $E$. Next, fix a split geometry structure on $E$ as a union of
contiguous blocks $B_k$, where each $B_k$ is either a split block, or
a thick block. 

Further, let $\partial B_k = S_{k-1} \cup S_k$ with $S_k$ the {\it
  upper boundary} and $S_{k-1}$ the {\it lower boundary}. Also let
  $S=S_0$ and $\sigma_i = \sigma_{i0}$. Let $\sigma_{ik}$ be the
  shortest closed curve in the split metric on $S_k$ (i.e. in the
  pseudometric obtained by electrocuting annular intersections of
  splitting Margulis tubes with the split level surface $S_k$) homotopic in $E$ to $\sigma_i = \sigma_{i0}$. Let
  $\overline{A_i} = D_i \bigcup_k \sigma_{ik} \subset M$ be the union of the disk
  $D_i$ and the quasi-annulus $\bigcup_k \sigma_{ik} $. Then any lift 
  $A_i$ of  $\overline{A_i}$ to $\til{M}$ is isometric to  $\overline{A_i}$ as $D_i$ is homotopically
  trivial and $\sigma_{ik}$ are all freely homotopic to $\sigma_i =
  \sigma_{i0} = \partial D_i$. 

We want to show that $A_i$ are quasiconvex in $(\til{M}, d_G)$ which in turn
is hyperbolic by Proposition \ref{graph_metric_hyperbolic}. 

\medskip

\noindent{\bf qi Rays}\\ Fix a $\sigma$ and the disk $D$ it bounds. Let $A = D \bigcup_k \sigma_{k} \subset M$, where $\sigma_k \subset S_k$.
Lift $\bigcup_k \sigma_{k}$ to the universal cover $\til{\til E}$ of $E$ such that any lift $\til{\sigma_k}$ lies in the universal cover $ \til{\til{S_k}}$ of $S_k$ (We are using this notation to distinguish
from lifts to $\til M$). Let $\lambda_k$ be any such lift $\til{\sigma_k}$.
We then  have the following from \cite{mahan-split}.

\begin{lemma}  \cite[Lemma 5.9]{mahan-split}
There exists $C \geq 0$ and for all $k\in \natls$, there exists $B(k)>0$  such that the following hold:
\begin{enumerate}
\item For $x_{k} \in \lambda_{k}$ there
exists
$x_{k-1} \in \lambda_{k-1}$ with $d_G (x_{k}, x_{k-1}) \leq C$ and $d_M (x_{k}, x_{k-1}) \leq B(k)$.
\item Similarly 
there
exists
$x_{k+1} \in \lambda_{k+1}$ with $d_G (x_{k}, x_{k+1}) \leq C$ and $d_M (x_{k}, x_{k+1}) \leq B(k)$.
\end{enumerate}
  Hence, for all $n$ and $x \in \lambda_n$,
 there exists a $C$-quasigeodesic ray $r$ (in the $d_G$-metric) 
such that $r(k) \in \lambda_k $ for all $k$
and $r(n) = x$. 
\label{dGqgeod0}
\end{lemma}

By construction of split blocks, $d_G (x_{i}, S_{i-1}) =
1$. Therefore inductively, $d_G (x_{i}, S_{j}) =
|i-j|$. Hence $d_G (x_{i}, x_{j}) \geq
|i-j|$. By construction, $d_G (x_{i}, x_{j}) \leq
C|i-j|$.

Hence, given $p \in \lambda_{i}$ the sequence of points $ x_{n}, n \in \mathbb{N} \cup \{ 0 \}$ 
 with $x_i = p$ gives by Lemma \ref{dGqgeod0} above, a quasigeodesic in
the $d_G$-metric. Such quasigeodesics shall be referred to as {\em $d_G$-quasigeodesic rays}. 

After projecting $\til E$ to $\til{M} \setminus \til{H}$ we have the following conclusion.

\begin{lemma}
There exists $C \geq 0$ and for all $k$ there exists $B_k$ satisfying the following:\\  For all $x_{ik} \in \sigma_{ik}$ there
exists
$x_{i,k-1} \in \sigma_{i,k-1}$ with $d_G (x_{ik}, x_{i,k-1}) \leq C$ and $d (x_{ik}, x_{i,k-1}) \leq B_k$.
\label{dGqgeod}
\end{lemma}

The following Corollary will turn out to be quite useful. Here we replace $A$ by the collection $A_i$ of annuli
constructed in the third paragraph of this section.

\begin{cor}  There exists $C \geq 0$ such that for all $k$ and all $x_{ik} \in \sigma_{ik}$
there exists $q \in \sigma_{i0}$ and a sequence of points $p = x_{ik},
\cdots, x_{i0}=q$ which is a quasigeodesic in $(\til{M}, d_G)$.
\label{dGqgeod1}
\end{cor}
\begin{proof}
By construction of split blocks, $d_G (x_{ik}, S_{k-1}) =
1$. 

Hence, given $p \in \sigma_{ik}$ the sequence of points $p = x_{ik},
\cdots, x_{i0}$ gives by Lemma \ref{dGqgeod} above, a quasigeodesic in
the $d_G$-metric lying entirely on $A_i$ joining $p$ to a point $q \in
D_i$. 
\end{proof}

We can choose a point $z_i \in D_i$ (quite arbitrarily) and extend any
 quasigeodesic constructed as above by adding on a path from $q$ to
 $z_i$ lying entirely in $D_i$ and having uniformly bounded length since
 $D_i$ has bounded diameter.

\begin{prop}
There exists $C_0 \geq 0$ such that each $A_j$ is $C_0$-quasiconvex in
 $(\til{M}, d_G)$.
\label{A-qc}
\end{prop}

\noindent {\bf Proof:} By Lemma  \ref{dGqgeod} and Corollary \ref{dGqgeod1}
above, it follows that  there exist $K \geq 1$ such 
that for any two points $p_1, p_2$ in $A_j$ there exist $K$-
quasi-geodesics $\gamma_1, \gamma_2$ to $z_j \in D_j$. By Proposition
 \ref{graph_metric_hyperbolic}
we also have that $(\til{M}, d_G)$ is hyperbolic. Hence any geodesic
$\alpha_i$ ($i = 1,2$) joining $p_i$ to $z_j$ lies in some $K_1$
neighborhood of $\gamma_i$. Further, by hyperbolicity of $(\til{M},
d_G)$, we conclude that a geodesic $\beta$ joining $p_1, p_2$ lies in a
$K_2$-neighborhood of $\alpha_1 \cup \alpha_2$. Hence, finally,
$\beta$ lies in a $(K_1 + K_2)$ neighborhood of $\gamma_1 \cup
\gamma_2 \subset A_j$. Choosing $C_0 = K_1 + K_2$, we are
through. $\Box$

The quasidisks constructed above have the following property.

\begin{lemma}  Let $M$ be a model manifold of split geometry. Let $H$ be a Scott core  of $M$
(augmented Scott core if $M$ has parabolics)
and $\{ D_i \}$ a maximal collection of disjoint compressing disks in $H$.
Then 
there exists a function $\Theta : \mathbb{N} \rightarrow \mathbb{N}$ satisfying
$\Theta (n) \rightarrow \infty$ as $n \rightarrow \infty$ such that for all  $o \in \til{H}$ the following holds.\\
Let $ D$ be a lift of one of the $D_i$'s to $\til M$ and let $A$  be the quasidisk in $\til M$ constructed
from $D$ as above.

Then $d_M(D, o) \geq n $ implies that $d_M(A, o) \geq \Theta (n)$.
\label{contlemma-qd} 
\end{lemma}

\begin{proof} By Lemma \ref{dGqgeod} and Corollary  \ref{dGqgeod1},  there exists $b_1, \cdots , b_k, \cdots$
 and $z \in D$ such that for all   $x_{k} \in \sigma_{k} \subset A$, 
 $d_M (x_{k}, z) \leq (b_1 + \cdots + b_k) = c_k$(say).
Hence $d_M (x_{k}, 0) \geq (n-c_k)$.

By uniform $\epsilon_0$-separatedness of split surfaces (Theorem \ref{wsplit}), 
 $d_M (x_{k}, 0) \geq k\epsilon_0$.
Hence $d_M (x_{k}, 0) \geq max( (n-c_k), k\epsilon_0)$.
Choosing $\Theta (n) $ to be the largest value of $k\epsilon_0$ such that $k\epsilon_0 \leq n-c_k$
we are done. \end{proof}

\subsection{Reduction Lemma} Before we get into the proof of the existence
of Cannon-Thurston maps, we recall some material from Section 6 of \cite{mahan-split} that will help
streamline the proof.

The next Lemma  allows us to apply the criterion for existence of Cannon-Thurston maps
in Lemmas \ref{contlemma} and \ref{crit-relhyp} to electro-ambient quasigeodesics in $\til M$
 rather
than hyperbolic geodesics in $\til N$.
Lemma \ref{contlemma2} below is a paraphrasing of  what Lemmas 6.8 and 6.10 of  \cite{mahan-split}   prove
in the context of this paper. (Note that though
 \cite[Lemmas 6.8 and 6.10]{mahan-split} are stated for simply degenerate surface groups,  the relevant parts of the
proofs only use incompressibility
of split components. This is pointed out in  \cite[Lemma 8.7]{mahan-split}.)

\begin{lemma}  \cite[Lemma 8.7]{mahan-split} Let $N$ be the convex core of a complete hyperbolic $3-$manifold $N^h$ minus a neighborhood
of the cusps. Equip each degenerate end with a split geometry structure such that each split component
is incompressible. Let $M$ be the resulting model of split geometry and $F: N\rightarrow M$ be the bi-Lipschitz
homeomorphism between the two. Let $\til{F}$ be a lift of $F$ to the universal covers.
Then for all $C_0 > 0$, and $o \in \til{N}$
there exists a function $\Theta : \mathbb{N} \rightarrow \mathbb{N}$ satisfying
$\Theta (n) \rightarrow \infty$ as $n \rightarrow \infty$ such that the following holds.\\
For any $a, b \in \til{N}\subset \til{N^h}$, let $\lambda^h $ be the hyperbolic geodesic in $\til{N^h}$ joining
them and let $\lambda^h_{thick} = \lambda^h \cap  \til{N}$. Similarly let $\beta_{ea}^h$ be an electro-ambient
$C_0-$quasigeodesic without backtracking in $\til M \subset \EE(\til{M}, \KK^\prime )$
 joining $\til{F} (a)$, $\til{F} (b)$. Let
$\beta_{ea} = \beta_{ea}^h \setminus \partial \til{M}$ be the part of $\beta_{ea}^h$ lying away from
the (bi-Lipschitz) horospherical boundary of $ \til{M}$.

Then $d_M(\beta_{ea}, \til{F} (o)) \geq n $ implies that $d_{\Hyp^3}(\lambda^h_{thick}, o) \geq \Theta (n)$.
\label{contlemma2}
\end{lemma}

Combining Lemma \ref{contlemma2} with  Lemma \ref{contlemma1}, we have the following.

\begin{cor}   Let $N$ be the convex core of a complete hyperbolic $3-$manifold $N^h$ minus a neighborhood
of the cusps. Equip each degenerate end with a split geometry structure such that each split component
is incompressible. Let $M$ be the resulting model of split geometry and $F: N\rightarrow M$ be the bi-Lipschitz
homeomorphism between the two. Let $\til{F}$ be a lift of $F$ to the universal covers.
Then for all $C_0 > 0$, and $o \in \til{N}$
there exists a function $\Theta : \mathbb{N} \rightarrow \mathbb{N}$ satisfying
$\Theta (n) \rightarrow \infty$ as $n \rightarrow \infty$ such that the following holds.\\
For any $a, b \in \til{N}\subset \til{N^h}$, let $\lambda^h $ be the hyperbolic geodesic in $\til{N^h}$ joining
them and let $\lambda^h_{thick} = \lambda^h \cap  \til{N}$. Similarly let $\beta_{ea}^h$ be an electro-ambient
$C_0-$quasigeodesic without backtracking in $\til M \subset \EE(\til{M}, \KK\bigcup\TT_l )$
 joining $\til{F} (a)$, $\til{F} (b)$. Let
$\beta_{ea} = \beta_{ea}^h \setminus \partial \til{M}$ be the part of $\beta_{ea}^h$ lying away from
the (bi-Lipschitz) horospherical boundary of $ \til{M}$.

Then $d_M(\beta_{ea}, \til{F} (o)) \geq n $ implies that $d_{\Hyp^3}(\lambda^h_{thick}, o) \geq \Theta (n)$.
\label{contlemma3}
\end{cor}

Again, combining Lemma \ref{contlemma2} with  Lemma \ref{contlemma-qd}, we have the following.

\begin{cor}  
Let $N$ be the convex core of a complete hyperbolic $3-$manifold $N^h$ minus a neighborhood
of the cusps. Equip each degenerate end with a split geometry structure such that each split component
is incompressible. Let $M$ be the resulting model of split geometry and $F: N\rightarrow M$ be the bi-Lipschitz
homeomorphism between the two. Let $\til{F}$ be a lift of $F$ to the universal covers. Let $H$ be a Scott core of $N$
and $\{ D_i \}$ a maximal collection of compressing disks in $H$.
Then 
there exists a function $\Theta : \mathbb{N} \rightarrow \mathbb{N}$ satisfying
$\Theta (n) \rightarrow \infty$ as $n \rightarrow \infty$ such that for all  $o \in \til{H}$ the following holds.\\
Let $ D$ be a lift of one of the $D_i$'s to $\til M$ and let $A$  be the quasidisk in $\til M$ constructed
from $D$ as above. For any $a,b \in A$, let $[a,b]_h$ be the hyperbolic geodesic in $\til{ N^h}$ joining
${\til{F}}^{-1}(a), {\til{F}}^{-1}(b)$ and let $[a,b] = [a,b]_h\cap \til{N}$.

Then $d_M(D, \til{F}(o)) \geq n $ implies that $d_{\Hyp^3}([a,b], o) \geq \Theta (n)$.
\label{contlemma-qd1}
\end{cor}

\begin{proof} By  Lemma \ref{contlemma-qd} there exists $z \in D$, a function  
 $\Theta_0 : \mathbb{N} \rightarrow \mathbb{N}$ satisfying
$\Theta_0 (n) \rightarrow \infty$ as $n \rightarrow \infty$ and electro-ambient quasigeodesics
$\beta_a, \beta_b$ in $(\til{M},d_G)$, joining $a$ to $z$ and $b$ to $z$ respectively, such that $d_M(D, \til{F}(o)) \geq n $ implies that 
$d_M(\beta_a \cup \beta_b, \til{F}(o)) \geq  \Theta_0 (n)$. 

Let $[a,z]_h, [b,z]_h$ be the hyperbolic geodesic in $\til N$ joining
${\til{F}}^{-1}(a), {\til{F}}^{-1}(b)$ respectively to ${\til{F}}^{-1}(z)$
and let $[a,z] = [a,z]_h\cap \til{N}$, $[b,z] = [b,z]_h\cap \til{N}$.

Then by  Lemma \ref{contlemma2} there exists a function  
 $\Theta_1 : \mathbb{N} \rightarrow \mathbb{N}$ satisfying
$\Theta_1 (n) \rightarrow \infty$ as $n \rightarrow \infty$ such that  
 $d_M(D, \til{F}(o)) \geq n $ implies that $d_{\Hyp^3}([a,z]\cup [b,z], o) \geq \Theta_1 (n)$.

Let $\delta >0$ be such that all geodesic triangles in ${\mathbb{H}}^3$ are $\delta -$thin.
Taking $\Theta (n)=\Theta_1 (n)-\delta$, it follows that 
$d_M(D, \til{F}(o)) \geq n $ implies that $d_{\Hyp^3}([a,b], o) \geq \Theta (n)$.
\end{proof}

\subsection{Admissible Quasigeodesics} \label{sec-adm}
We shall need a collection of paths consisting of horizontal and vertical segments
approximating electro-ambient quasigeodesics. We shall call these admissible quasigeodesics.
Let $M$ be a model manifold each of whose ends is equipped with a split geometry structure such that all split
components are incompressible. Recall that
each thick block and each split
block in $M$ is  homeomorphic to   a product $\Sigma_i^s \times I$. We fix such a product structure
for each block.
Let $t_i = sup\{ length(\{x\} \times I): x \in \Sigma_i^s\}$ be the {\it thickness} of the $i-$th block.

Recall that $F: N \rightarrow M$ is a bi-Lipschitz homeomorphism from a hyperbolic
manifold $N$ (minus cusps) to $M$ and let $\til F$ denote its lift to the universal cover.

An {\bf elementary  admissible path} in $\til M$ is one of the following: 

\noindent Type 1: A `horizontal' geodesic in the intrinsic path metric on some lift $\til{\Sigma_i^s}$ of a split  surface to $\til M$.\\
Type 2: A `vertical' path of the form $x \times I$ (with respect to the fixed product
structure above) in the lift to   $\til M$ of
either a thick block or a split block. \\

Let $B_i$ be a thick block of $M$ and  let $S_i^s, S_{i+1}^s$ be its horizontal boundary components.
Recall that a product structure $B_i = S \times I$ has been fixed.
Let $\mu = [a,b]$ be a geodesic in $\til{B_i} \subset \til{M}$ such that its end-points 
$a, b$ lie on the horizontal boundary components. Let $P(\mu)$ denote the projection of $\mu$ 
onto the horizontal boundary component ( $S_i^s$ or $ S_{i+1}^s$) containing $a$. If $b$ belongs to
the same  horizontal boundary component  as $a$, define $\mu_{adm} = P(\mu)$. Else define
$\mu_{adm} = P(\mu)\cup \{b\}\times I$, where $\{b\}\times I$ is the elementary vertical path
through $b$. $\mu_{adm}$ will be called the {\bf admissible quasigeodesic} corresponding
to $\mu$ in the thick block  $\til{B_i}$.

Let $\BB$ denote the collection of thick blocks.

\begin{defn} \label{adm-def} Let $\beta_{ea}$ be an  electro-ambient
$C_0-$quasigeodesic ( without backtracking), such that it enters or leaves  split components at split level surfaces. Then 
an {\bf admissible quasigeodesic} $\beta_{adm}$  in $\til M$ corresponding to   $\beta_{ea}$   in $\til M \subset \EE(\til{M}, \KK\bigcup\TT_l )$ is a path such that
 $\beta_{adm} \cap (\til{M} \setminus \bigcup_{K \in \KK} \til{K}
\cup  \bigcup_{B\in \BB} \til{B} ) = \beta_{ea} \cap (\til{M} \setminus 
\bigcup_{K \in \KK} \til{K}\cup  \bigcup_{B\in \BB} \til{B} )$. 

Further, 
for each $ \til{K}$, $\beta_{adm} \cap \til{K}$ is  a union of elementary  admissible paths with
disjoint interiors such that
\begin{enumerate}
\item $\beta_{adm} \cap \til{K}$  has at most one vertical
path of type (2) above.
\item  for any connected horizontal boundary 
component $\til{\Sigma_0^s}$ of  $\til K$,
$\beta \cap \til{\Sigma_0^s}$ has at most one `horizontal' geodesic of type  (1) above.
\end{enumerate}

Finally for each $B\in \BB$, $\beta_{adm} \cap \til{B}$ is the admissible quasigeodesic
corresponding to $\beta_{ea} \cap \til{B}$ in $\til{B}$.
\end{defn}

Definition \ref{adm-def} allows us to replace an electro-ambient
quasigeodesic by a path with greater control. We shall be concerned with electro-ambient quasigeodesics $\beta_{ea}$ in $\til M$ starting and ending at 
points in $\til H$. Any such electro-ambient quasigeodesic has a representative entering and leaving split components at split level surfaces.
This follows from the observation that any electro-ambient quasigeodesic necessarily 
enters and leaves split blocks as well as thick blocks
along (horizontal) split level surfaces except at most 
for maximal pieces that have end-points  in one of the thin tubes in $\TT_l$. Hence the restriction  on
electro-ambient quasigeodesics $\beta_{ea}$ given by (the first sentence of) Definition \ref{adm-def} may be regarded just as  a choice
of a representative of $\beta_{ea}$ or as
a mild  normalization
condition. The construction of the admissible quasigeodesic $\beta_{adm}$ corresponding to $\beta_{ea}$  changes this representative of 
$\beta_{ea}$ by replacing the intersection of $\beta_{ea}$ with each thick block $\til B$ or split component $\til K$ by one horizontal
and at most one vertical piece.
 The reader should thus think of an admissible quasigeodesic $\beta_{adm}$ corresponding to $\beta_{ea}$ as an  
electro-ambient quasigeodesic
\begin{enumerate}
\item with the same entry and exit points as $\beta_{ea}$ into either thick blocks (elements
of $\til B$) or split components (elements of $\til K$).  
\item if the entry and exit points $p,q$ of $\beta_{ea}$ into some $\til B$ (or $\til K$) lie on the same horizontal level, join $p,q$
 by a horizontal geodesic
in the corresponding level surface to obtain the corresponding piece of $\beta_{adm}$.
\item if the entry and exit points $p,q$ of $\beta_{ea}$ into some $\til B$ (or $\til K$) lie on levels $i$ and $i+1$, say, then project $q$
(using the product structure of the block $B$ or $K$) to $P(q)$ on level $i$. Join $p, P(q)$ by a horizontal 
geodesic
in the $i$-th level surface and then follow it by the vertical segment joining $P(q)$ to $q$.
\end{enumerate}
 The choice of an admissible quasigeodesic corresponding to an electro-ambient quasigeodesic is not unique. In each piece of the third type above, we can
 also choose
 the vertical segment  at $p$ and then follow by a horizontal 
geodesic
in the $(i=1)$-th level surface. The ambiguity is bounded however in a sense made precise in Lemma \ref{contlemma4} below.

The next Lemma  allows us to apply the criterion for existence of Cannon-Thurston maps
in Lemma \ref{contlemma} to admissible quasigeodesics in $\til M$
 rather
than electro-ambient quasigeodesics in $\til M$. The proof of
Lemma \ref{contlemma4} is exactly like Lemma 6.5 of  \cite{mahan-split}   and we omit it here
(see also the proof of Lemma \ref{contlemma1} above).

\begin{lemma}   Let $N$ be the convex core of a complete hyperbolic $3-$manifold $N^h$ minus a neighborhood
of the cusps. Equip each degenerate end with a split geometry structure such that each split component
is incompressible. Let $M$ be the resulting model of split geometry.
Then for all $C_0 > 0$, and $o \in \til{M}$
there exists a function $\Theta : \mathbb{N} \rightarrow \mathbb{N}$ satisfying
$\Theta (n) \rightarrow \infty$ as $n \rightarrow \infty$ such that the following holds.\\
For any $a, b \in \til{M}$, let $\beta_{ea}^h $  be an electro-ambient
$C_0-$quasigeodesic without backtracking in $\til M$ joining $a, b$. Let
$\beta_{ea} = \beta_{ea}^h \setminus \partial \til{M}$ be the part of $\beta_{ea}^h$ lying away from
the (bi-Lipschitz) horospherical boundary of $ \til{M}$. Again, let $\beta_{adm}^h $  be the admissible
quasigeodesic corresponding to  $\beta_{ea}^h $ and let 
$\beta_{adm} = \beta_{adm}^h \setminus \partial \til{M}$ be the part of $\beta_{adm}^h$ lying away from
the (bi-Lipschitz) horospherical boundary of $ \til{M}$.

Then $d_M(\beta_{ea}, o) \geq n $ implies that $d_M(\beta_{adm}, o) \geq \Theta (n)$.
Conversely, $d_M(\beta_{adm}, o) \geq n$  implies that $d_M(\beta_{ea}, o) \geq  \Theta (n)$.
\label{contlemma4}
\end{lemma}

\subsection{Cannon-Thurston Maps for Free Groups}
We are now in a position to prove the existence of Cannon-Thurston maps for arbitrary finitely generated
Kleinian groups. In this subsection we shall describe the proof for handlebody groups where the book-keeping is minimal.
In the next subsection we shall indicate the modifications necessary for arbitrary finitely generated
Kleinian groups.
We identify $H$ with its bi-Lipschitz image $F(H)$ contained in $M$ under the bi-Lipschitz homeomorphism $F: N\rightarrow M$ from the hyperbolic manifold
$N$ to the model manifold $M$.

We now want to show that if $\lambda = [a,b]$ is a geodesic in the
intrinsic metric on $\til{H}$ joining $a, b \in \til{H}$, and lying
outside a large ball about a fixed reference point $p \in \til{H}
\subset \til{M}$, then  the (bi-Lipschitz) hyperbolic geodesic
$\lambda^h$ joining $a, b \in \til{M}$ also lies outside a large ball
about $p$ in $\til{M}$. This would guarantee the existence of a
Cannon-Thurston Map by Lemma \ref{contlemma}.

To fix notation, 
let $Q$ be a free Kleinian group without parabolics. Let $N = {{\mathbb{H}}^3}/Q$, $M=F(N)$
and $H$ a compact (Scott) core of $M$. $\til{H}$ with its intrinsic
metric is quasi-isometric to the Cayley graph $\Gamma_Q$ and so its
intrinsic boundary may be identified with the Cantor set $\partial Q$
thought of as the Gromov boundary of $\Gamma_Q$. Let $\hhat{H}$ and
$\hhat{M}$ denote the compactifications by adjoining $\partial Q$ and
the limit set $\Lambda_Q$ to $\til{H}$ and $\til{M}$  respectively.

\begin{theorem} {\bf Cannon-Thurston for Free Groups} The inclusion
  $i: \til{H} \rightarrow \til{M}$ extends continuously to a map
  $\hat{i}: \hhat{H} \rightarrow \hhat{M}$. 
\label{ct-free}
\end{theorem}

\begin{proof} Let $p \in  \til{H}$ be a base-point, and
 $\lambda  = [a,b]$ be a geodesic in the
intrinsic metric on $\til{H}$ and 
$\lambda^h$ be the (bi-Lipschitz) hyperbolic geodesic
 joining its end-points in $ \til{M}$.  By Lemma \ref{contlemma} it suffices to show
that if $\lambda$ lies outside a large ball about $p$ in $\til{H}$,
then $\lambda^h$ lies outside a large ball about $p$ in
$\til{M}$.

Suppose that $\lambda$  lies outside an $n$-ball about $p$ in $\til H$, i.e. $d_{\til{H}}(\lambda, p) \geq n$.
Let $\{ D_i \}$ be a finite collection of compressing disks in $H$ such that each component
of $\partial H \setminus \bigcup_i \partial D_i$ is a pair of pants.

Since each (lift of) $D_i$ separates $\til{H}$ and since $\lambda$ lies
outside a large $n-$ball about $p$ in $\til{H}$, we conclude that there
exists such a lift $D$ lying outside an $m=m(n)$- ball about $p$ in $\til{M}$, (i.e. $d_{\til{M}}(D , p) \geq m(n)$) and
that $\lambda$ lies in the component of $\til{H}\setminus D$ not containing $p$, where
$m(n) \rightarrow \infty$ as $n \rightarrow \infty$. Let $A=D\cup\bigcup_i \sigma_i$ be the quasidisk containing $D$
constructed in Section \ref{qd},
where $\sigma_i$ is a closed curve on the lift $\til{\Sigma_i^s}$ of the $i-$th split surface $\Sigma_i^s$
to $\til M$. Also, let $\til{\Sigma_i^s} \setminus \sigma_i = {\til{\Sigma_i^s}}_+ \cup {\til{\Sigma_i^s}}_-$
where $ {\til{\Sigma_i^s}}_+ , {\til{\Sigma_i^s}}_-$ are the two components
of $\til{\Sigma_i^s} \setminus \sigma_i$. Similarly, let $\til{H}\setminus D={\til{H}}_+\cup{\til{H}}_-$,
where $ {\til{\Sigma_0^s}}_+\subset \partial {\til{H}}_+$ and
$ {\til{\Sigma_0^s}}_-\subset \partial {\til{H}}_-$. Assume without loss of generality that
$\lambda \subset {\til{H}}_+$ and $p \in {\til{H}}_-$. Let $ {\til{M}}_+= {\til{H}}_+\cup
\bigcup_i {\til{\Sigma_i^s}}_+$
and $ {\til{M}}_-= {\til{H}}_-\cup \bigcup_i {\til{\Sigma_i^s}}_-$.  Also let ${\til{M}}_H=
 \bigcup_i {\til{\Sigma_i^s}}$ be the union of all the horizontal split surfaces lifted to $\til M$.

Let $\alpha$ be an electro-ambient quasigeodesic joining the end-points of  $\lambda$
 in $ \til{M}$ and $\beta$ be an admissible  quasigeodesic corresponding to $\alpha$.

Since $\beta$ is admissible, it consists of horizontal and vertical pieces.
Two cases arise:\\
a)  $\beta\cap {\til{M}}_H \subset {\til{M}}_+$\\
b) $\beta\cap ({{\til{M}}_+})^-\cap  ({{\til{M}}_-})^- \neq \emptyset$.

Roughly speaking Cases (a) and (b) correspond respectively
 to the cases where $\beta$ does not or does intersect $A$ coarsely.

\noindent {\bf Case a:}  $\beta\cap {\til{M}}_H \subset {\til{M}}_+$\\
By Corollary \ref{contlemma3} and Lemma \ref{contlemma4} it suffices to show that
there exists a function $\Theta : \mathbb{N} \rightarrow \mathbb{N}$ satisfying
$\Theta (n) \rightarrow \infty$ as $n \rightarrow \infty$ such that the following holds.\\
$d_{\til{H}}(\lambda , p) \geq n$ implies that $d_M(\beta, p) \geq \Theta (n)$.

The existence of such a function $\Theta : \mathbb{N} \rightarrow \mathbb{N}$ follows exactly as in Lemma 
\ref{contlemma-qd}. We briefly recount the proof. It suffices to prove that if 
$d_{\til{M}}(D , p) \geq m(n)$, then  $d_{\til{M}}(A , p) \geq \Theta(n)$. As in the proof of  Lemma 
\ref{contlemma-qd}, there exist constants $b_1, \cdots, b_k, \cdots$ (depending only on the geometry of $E$ and the split geometry model)
and a $z \in D$
such that for any $x_k \in \sigma_k (\subset A)$, we have $d_M(x_k,z) \leq b_1 + \cdots + b_k = c_k$ (say). Hence $d_M(x_k,p) \geq m(n) - c_k$.
Also, using uniform separatedness, $d_M(x_k,z) \geq k \epsilon_0$.
Therefore there exists a proper function $\Theta_0 : \mathbb{N} \rightarrow \mathbb{N}$ such that $d_M(A,p) \geq \Theta_0(n)$.
Since $\beta$ is disjoint from $A$ and is admissible, it follows that every horizontal segment of $\beta$  is at distance at least $\Theta_0(n)$ from $p$.

Any vertical segment $\tau_k$ of $\beta$ necessarily joins a pair of points $y_k, y_{k+1}$, where $y_k, y_{k+1}$ lie on
the top and bottom horizontal boundaries of the same block (thick or split) with $y_k$ (resp. 
$y_{k+1}$) lying on the same horizontal split level surface as $\sigma_k$  (resp. 
$\sigma_{k+1}$). Further, the vertical segment has length at most $C_k$ -- the thickness
of the  $k-$th block. It follows that any point on $\tau_k$ lies at 
distance at least $\Theta_0(n) - C_k$ from $p$. Again, there exists $\epsilon_0$ (a lower bound on the separation between successive horizontal levels)
such that any point on $\tau_k$ lies at 
distance at least $k\epsilon_0$ from $p$. Choosing $\Theta(n) = max\{ \Theta_0(n) - C_k,  k\epsilon_0  \} $, we are done.

\noindent {\bf Case b:}  $\beta\cap ({{\til{M}}_+})^-\cap  ({{\til{M}}_-})^- \neq \emptyset$.\\
We shall say that $\beta$ crosses $A$ at $x$ if either $x \in \beta \cap A$, or if there exists
a vertical elementary admissible subpath $x \times I \subset \beta$, such that 
either $(x,0) \in {\til{M}}_+$ and
$(x,1) \in {\til{M}}_-$ or $(x,1) \in {\til{M}}_+$ and
$(x,0) \in {\til{M}}_-$.

Let $r, q$ be the first and last points at which $\beta$ crosses $A$.  Let $\beta_{ar},
\beta_{qb}$ be the subpaths of $\beta$ joining $a, r$ and $b, q$ respectively.
Then again, as in Case (a) above, there exists a function $\Theta : \mathbb{N} \rightarrow \mathbb{N}$ satisfying
$\Theta (n) \rightarrow \infty$ as $n \rightarrow \infty$ such that 
$d_{\til{H}}(\lambda , p) \geq n$ implies that $d_M(\beta_{ar}\cup
\beta_{qb}, p) \geq \Theta (n)$. Hence by Corollary \ref{contlemma3} and Lemma \ref{contlemma4}
it follows that  there exists a function $\Theta_2 : \mathbb{N} \rightarrow \mathbb{N}$ satisfying
$\Theta_2 (n) \rightarrow \infty$ as $n \rightarrow \infty$ such that
$d_{\til{H}}(\lambda , p) \geq n$ implies that $d_M(\mu_{ar}\cup
\mu_{qb}, p) \geq \Theta_2 (n)$, where $\mu_{ar}$ (resp. $\mu_{qb}$) are the (bi-Lipschitz)
hyperbolic geodesics in $\til M$ joining $a, r$ and $b, q$ respectively.

If $r, q \in A$, then, since $d_M(D,p) \geq m(n)$,  Corollary \ref{contlemma-qd1} ensures the existence of a proper
 function $\Theta_3 : \mathbb{N} \rightarrow \mathbb{N}$ 
(i.e. $\Theta_3 (n) \rightarrow \infty$ as $n \rightarrow \infty$) such that
$d_{\til{H}}(\lambda , p) \geq n$ implies that $d_M(\mu_{rq}, p) \geq \Theta_3 (n)$, where $\mu_{rq}$
is the (bi-Lipschitz)
hyperbolic geodesic in $\til M$ joining $ r$ and $ q$. Hence by $\delta -$hyperbolicity
of $\til M$,  $d_M(\mu_{ab}, p) \geq min(\Theta_2 (n), \Theta_3 (n)) - 2 \delta$ and we are done.

Else, let $r \in \til{\Sigma_i^s}$ and $q \in \til{\Sigma_j^s}$. There exists $r'\in \til{\Sigma_{i+1}^s}$, such that $r,r'$ lie on a vertical
segment in $\beta$. Further they lie on opposite ($+$ and $-$) sides of $A$. Similarly, there 
exists $s'\in \til{\Sigma_{j-1}^s}$, such that $s,s'$ lie on a vertical
segment in $\beta$ and on opposite ($-$ and $+$) sides of $A$.  Hence, (depending on thickness of the $m$-th block)
there exist $C(m), m \in \mathbb{N}$
such that $d_M(r, A) \leq C(i)$ and $d_M(q, A) \leq C(j)$. Choose $r_1,
q_1$ in $\sigma_i, \sigma_j$ respectively such that $d_M(r, r_1) \leq C(i)$ and $d_M(q, q_1) \leq C(j)$.

Then, by Corollary \ref{contlemma-qd1},
 there exists a function $\Theta_4 : \mathbb{N} \rightarrow \mathbb{N}$ satisfying
$\Theta_4 (n) \rightarrow \infty$ as $n \rightarrow \infty$ such that 
$d_M(\mu_{r_1q_1}, p) \geq \Theta_4 (n)$, where $\mu_{r_1q_1}$
is the (bi-Lipschitz)
hyperbolic geodesic in $\til M$ joining $ r_1$ and $ q_1$. The existence of $C(m),, m \in \mathbb{N}$ now guarantees the 
existence of a proper function $\Theta_5 : \mathbb{N} \rightarrow \mathbb{N}$ such that 
$d_M(\mu_{rq}, p) \geq \Theta_5 (n)$ (reprising, for instance, the argument in the last paragraph of Case (a) above). Hence by $\delta -$hyperbolicity
of $\til M$ again,  $d_M(\mu_{ab}, p) \geq min(\Theta_2 (n), \Theta_5 (n)) - 2 \delta$ and we are through.
\end{proof}

\subsection{Finitely Generated Kleinian Groups}
In this subsection, we indicate the modifications necessary in the proof of  Theorem \ref{ct-free}
 to prove the analogous theorem for finitely
generated Kleinian groups.

Let $N_{gf}$ denote the augmented Scott core of  $N^h
= {\Hyp}^3 / G$. Let $i: N_{gf} \rightarrow
N^h$ be the natural inclusion map. Thurston showed that there exists a geometrically finite manifold admitting a strictly type-preserving
 homotopy equivalence with $N^h$. The convex core of such a manifold admits a proper homeomorphism to $N_{gf}$. Thus,  $N_{gf}$ may be thought of as
 the  convex core of  a geometrically finite manifold admitting a strictly type-preserving
 homotopy equivalence with $N^h$.
  Let $H$ be $N_{gf}$ with open neighborhoods of cusps removed.
 Let $N$ be $N^h$ with open neighborhoods of cusps removed. Then $\til H$ is strongly hyperbolic relative to its horospheres and $\til N$ 
is strongly hyperbolic relative to its horospheres. Let $\widehat{H}$ and $\widehat{N}$ denote their relative hyperbolic compactifications. Note that 
$\widehat{H}=\widehat{N_{gf}}$, where $\widehat{N_{gf}}$ is the Gromov compactification of the hyperbolic space $\widetilde{N_{gf}}$. Similarly,
$\widehat{N}=\widehat{N^h}$, where $\widehat{N^h}$ is the Gromov compactification of the hyperbolic space $\widetilde{N^h}$.
 Let $\til{i}: \til{H} \rightarrow \til{N}$ indicate
the lift of $i$. Let $M$ denote the model manifold for $N$ and let $\widehat{M}$ denote the relative hyperbolic
compactification of $\til M$. We identify $H$ with its bi-Lipschitz image in $M$ under the bi-Lipschitz homeomorphism
$F:N\rightarrow M$. Let $d_G$ be the graph metric on $\til M$ equipped with a split geometry structure where all split components are incompressible.

First, suppose that
 $H$ has incompressible boundary as a pared manifold. Then Theorem \ref{ptpre-inde} shows that a Cannon-Thurston map exists for $\til{i} : \til{H} \rightarrow \til{M}$. The point pre-image
description is also furnished by Theorem \ref{ptpre-inde}.

Else $H$ may be decomposed as the disk-connected  sum (or boundary-connected sum) $\#_{i=1, \cdots, r+s} H_i$ of
$H_1$, $H_2 \cdots  H_{r+s}$, where 
\begin{enumerate}
\item $H_{r+1}, \cdots , H_{r+s}$ are pared manifolds with incompressible boundary.
\item $H_1, \cdots, H_r$ are handlebodies such that  in the boundary connected sum decomposition of $H$, no two are connected
to the same boundary component of $\#_{i=r+1, \cdots, r+s} H_i$. (This ensures that the boundary connected sum decomposition is minimal.)
\end{enumerate}
Let $D'_{1}, \cdots , D'_{r+s-1}$ be the compressing disks obtained from the above boundary connected sum decomposition.
Next for each handlebody $H_i$ ($i= 1, \cdots, r$), we choose a minimal collection of disjoint non-separating compressing
disks such that their complement in
 $H_i$  is a ball. Taking the union of $D'_{1}, \cdots , D'_{r+s-1}$ with all these compressing disks for $H_i$ ($i= 1, \cdots, r$),
we obtain a collection
 $D_i, i = 1 \cdots m$ of  compressing disks.   Since
 $m \geq 1$, there is at least one compressing disk. Note that the collection  $\partial D_i$, $i = 1 \cdots m$ forms a maximal
collection of homotopically distinct compressible simple closed curves on $\partial H$.

\begin{theorem} {\bf Cannon-Thurston for Kleinian Groups} Let $H, M, N, N_{gf}, N^h$ be as above. 
Further suppose that each degenerate end of $M$ admits a Minsky model.
The inclusion
  $\til{i}: \til{H} \rightarrow \til{M}$ extends continuously to a map
  $\hat{i}: \hhat{H} \rightarrow \hhat{M}$   between the relative hyperbolic compactifications. Equivalently, 
 $\til{i}: \til{N_{gf}} \rightarrow \til{N^h}$ extends continuously to a map
 $\hat{i}: \hhat{N_{gf}} \rightarrow \hhat{N^h}$   between the  hyperbolic compactifications. 
\label{ct}
\end{theorem}

\noindent {\bf Proof:} From  Propositions
\ref{incompressible}, \ref{graph_metric_hyperbolic} and \ref{A-qc}, we
can construct quasidisks $A_i$ corresponding to $D_i$ as before and
lift them to $\til{M}$ (after partially electrocuting $\mathbb{Z}$-cusps if any).

Now, let $\lambda$ be a geodesic segment in $\til{H}$ lying outside a
large ball $B_N(p)$ for a fixed reference point $p$. $\lambda$ may be
decomposed into (at most) three pieces $\lambda_-$, $\lambda_0$ and
$\lambda_+$ as follows. \\
\begin{enumerate}
\item The middle piece
$\lambda_0$ does not intersect any of the (lifts of the) compressing
disks $D_i$ in the interior. (We thus allow for the cases where
$\lambda_-$ and/or $\lambda_+$ are empty.) 
\item  
the common end-point  $\lambda_- \cap \lambda_0$ lies on some
$D_i$. The same is demanded of $\lambda_0 \cap \lambda_+$. 
\item the point $q$ on $\lambda$ nearest to $p$ lies on $\lambda_0$. 
\item If $\lambda$ intersects exactly one disk $D$, then $\lambda_0$ is defined to be
the piece of $\lambda$ ending on $D$ and containing $q$; the other piece being designated $\lambda_-$ or $\lambda_+$.
\end{enumerate}

We shall consider a sequence of $\lambda$'s converging to a point $\xi$ on the boundary of $\hhat{H}$ (or $\hhat{N_{gf}}$)
in the Hausdorff topology. (Equivalently, the sequence of end-points of the $\lambda$'s converge to $\xi$.)
Such a sequence must necessarily lie outside larger
 and larger balls $B_n(p)$ about $p$.
 Two cases arise.

{\bf Case A:}  For a sequence of $\lambda$'s lying outside larger
 and larger balls $B_n(p)$ about $p$, and converging to $\xi$, the sequence of $\lambda$'s
eventually lies outside any fixed component of $\til{H} \setminus \cup_i D_i$.
 This is exactly the same case as in the proof of Theorem
 \ref{ct-free}.
The same proof goes through by Corollary \ref{contlemma-qd1}.

{\bf Case B:} Either $[p,q]$ does not intersect any $D_i$ or 
there is (up to subsequencing) a fixed disk $D_i$ that is the
last disk that $[p,q]$ intersects. Since $D_i$ is of uniformly bounded
diameter, we may shift our base point to a point $p'$
in the component $\til{H_i}$ which is the lift of $H_i$  containing
$q$. In this case, there exists a fixed $n_0$ such that $\lambda$ lies
outside $B_{(n-n_0)}(p')$. By shifting origin, we rewrite $p'$ as $p$
  and 
$(n-n_0)$ as $n$.

\noindent{\bf Step 1:}
Now, $\lambda_0 \subset \til{H_i} \subset \til{H}$  as it does not meet any
disk $D_i$ in its interior. Since $H_i$ is either a  handlebody
without parabolics, or a  pared manifold with incompressible boundary, then by
Theorem \ref{ct-free} or Theorem \ref{ptpre-inde}
 respectively, a Cannon-Thurston map exists for the inclusion  $\til{H_i} \subset \til{M}$. By (the necessity part of)  Lemma \ref{contlemma} and Lemma \ref{crit-relhyp}, it follows that   the
hyperbolic geodesic $\lambda_0^h$ joining the end-points of
$\lambda_0$ in $\til{N^h}$ lies outside a large ball about $p$. Thus,
there exists $m_1(n) \rightarrow \infty$ as $n \rightarrow \infty$
such that $\lambda_0^h$ lies outside a ball of radius $m_1(n)$ about
$p$ in $\til{N^h}$. 

\medskip

\noindent{\bf Step 2:} If $\lambda_+$ (or $\lambda_-$) is non-empty,
then $\lambda_-$ (or $\lambda_+$) is separated from $p$ by a disk $D_i
\subset \til{H}$ lying outside $B_n(p)$. \\

\medskip

\noindent {\bf Note:} Strictly speaking some uniformly bounded pieces of $\lambda_+$ (or $\lambda_-$) close to the intersection point 
with $\lambda$ could enter $\til{H_i}$.
To see the existence of  a uniform bound, 
note first that $\til H$ is quasi-isometric to (the Cayley graph of) $\pi_1(M) (= \pi_1(H))$. Then the disks $D_i$ correspond to splittings of the group $\pi_1(H)$. Geodesic words in $\pi_1(H)$ cannot `backtrack', i.e. in the tree of
spaces description of $\til H$ corresponding to splittings by $D_i$, such paths cannot re-enter a vertex space after
 leaving it. Uniform boundedness now follows from the quasi-isometry between 
$\til H$ and (the Cayley graph of) $\pi_1(M) (= \pi_1(H))$.

 However they can be replaced first by paths lying entirely on $D_i$ and then pushed out of 
$\til{H_i}$ altogether. Thus, replacing the geodesic segments $\lambda_-$ (or $\lambda_+$)  by quasigeodesics if necessary, we can ensure
that
$\lambda_-$ (or $\lambda_+$) is separated from $p$ by a disk $D_i
\subset \til{H}$ lying outside $B_n(p)$. We shall ignore this  mild technicality.

\bigskip

Recall that $M$ is a bi-Lipschitz model for $N$, which in turn is $N^h$ with cusps removed. Also recall that $d_G$ is the graph metric on
$\til M$. Then the quasidisk $A_i$ is
quasiconvex in $(\til{M}, d_G)$ and lies outside a large ball of radius $m_2(n)$ about
$p$, where $m_2(n) \rightarrow \infty$ as $n \rightarrow \infty$. 
Let $\lambda_{+}^h$ (resp. $\lambda_{-}^h$) be the hyperbolic geodesics joining the end-points
of $\lambda_{+}$ (resp. $\lambda_{-}$).
Again, by constructing admissible paths and electro-ambient
quasigeodesics as in the proof of Theorem \ref{ct-free}, we obtain a
new function $m_3(n)$ such that
$m_3(n) \rightarrow \infty$ as $n \rightarrow \infty$ and so that the
hyperbolic geodesics $\lambda_{-}^h$ or  $\lambda_{+}^h$ lie outside a
ball of radius $m_3(N)$ about $p$ in $\til{N^h}$.\\

\noindent{\bf Step 3:} Therefore $\lambda_{-}^h \cup \lambda_{0}^h \cup
\lambda_{+}^h$ lies outside a ball of radius $m_4(n) = min \{ m_1(n),
m_3(n) \}$. Finally, since $\til{N^h}$ is $\delta$-hyperbolic, the hyperbolic
geodesic $\lambda^h$ joining the end-points of $\lambda$ lies outside
a ball of radius $m(n) = m_4(n) - 2 \delta$ about $p$. Also,
$m(n) \rightarrow \infty$ as $n \rightarrow \infty$. Therefore, by
Lemma \ref{contlemma} or \ref{crit-relhyp}, it follows that 
 the inclusion
  $\til{i}: \til{H} \rightarrow \til{M}$ extends continuously to a map
  $\hat{i}: \hhat{H} \rightarrow \hhat{M}$ or equivalently that  $\til{i}: \til{N_{gf}} \rightarrow \til{N^h}$ extends continuously to a map
 $\hat{i}: \hhat{N_{gf}} \rightarrow \hhat{N^h}$. This concludes the proof. $\Box$

\medskip

\begin{remark}
As mentioned in Remark \ref{minapp}, the hypothesis that each degenerate end of $M$ admits a Minsky model is expected
to be superfluous and is established in Appendix \ref{app} when $M$ has no parabolics.
\end{remark}

Let $\hhat{{G}_F}$ denote the Floyd compactification of a group $G$
(See \cite{Floyd}). McMullen conjectured in \cite{ctm-locconn} that
there exists a continuous extension of $i:\Gamma_G \rightarrow
\til{M}$ to a map from $\hhat{G_F}$ to $\hhat{M}$. It was shown by
Floyd in \cite{Floyd} that there is a continuous map from $\hhat{G_F}$
to $\hhat{H}$. Combining this with  Theorem \ref{ct} above for Kleinian groups with parabolics, we  get a 
proof of the following, which proves McMullen's conjecture.

\begin{theorem} Let $G$ be any finitely generated Kleinian group  and $M =
  {\Hyp}^3/G$, Further suppose that each degenerate end of $M$ admits a Minsky model (for instance
if $M$ has no parabolics). Then there is a
  continuous extension $\hat{i}: \hhat{G_F} \rightarrow \hhat{M}$.
\label{conj-ctm}
\end{theorem}

\section{Point pre-images of the Cannon-Thurston Map }\label{ptpreim}
In this section, we first
determine the pre-images of multiple points under the Cannon-Thurston map for degenerate free Kleinian groups $G$. 
We then indicate the modifications necessary to extend the results to arbitrary
finitely generated Kleinian groups. This is done for two reasons: \\
a) Otal had specifically conjectured \cite{otal-thesis} the structure of the Cannon-Thurston maps we prove for handlebody groups.\\
b) The additional work required for arbitrary finitely generated Kleinian groups involves extra book-keeping
with respect to a finite family of ends.\\
The extra generality at this stage would tend to clutter up the exposition.

We shall not have need to distinguish between the hyperbolic manifold and its bi-Lipschitz model any longer and will denote
the manifold by $M$.

We set up some notation for the purposes of this section. Let $G$ be a free degenerate Kleinian group without parabolics. Suppose that $G$ is not geometrically
finite. Let $M = {\mathbb{H}}^3/G$ be the quotient manifold. Note that the limit set of $G$ is all of the sphere at infinity. Hence $M$ is its own convex core.
Let $H$ be a compact core of $M$. $H$ is a handlebody whose inclusion into $M$ induces a homotopy equivalence. In fact, $M$ deformation
retracts onto $H$. Then $\til H$ is embedded in $\til{M} = {\mathbb{H}}^3$. Let $\Gamma$ denote the Cayley graph of $G$ with respect to some finite generating set of
$G$. Assume that $\Gamma$ is equivariantly 
embedded in $\til H$ with edges being mapped to geodesic segments.
 Let $S$ denote the boundary surface of $H$. We assume that the ending lamination $\Lambda_{EL}$
is a geodesic lamination on $S$ equipped with some (any) hyperbolic metric. This is well-defined only
up to Dehn twists along simple closed curves in $S$ that bound
disks in $H$ and gives a well-defined ending lamination in the Masur domain by Theorem \ref{elinmd}. To make this explicit, we denote the
ending lamination in the Masur domain by $\Lambda_{ELH}$. $M \setminus Int( H)$ is homeomorphic to
$S \times [0, \infty )$ and is bi-Lipschitz homeomorphic to an end $M_S$ of a simply degenerate hyperbolic manifold  without accidental parabolics
\cite{bowditch-model} \cite{minsky-elc3}. Thus 
$S \times [0, \infty ) \subset M$ equipped with its {\it intrinsic} path metric is bi-lipschitz homeomorphic to  $M_S$. We shall have need to pass
interchangeably between these two below.

\subsection{EL leaves are CT leaves}

Let $i : \til H \rightarrow \til{M}$ denote the inclusion. 
 Let $\partial i$ denote the continuous extension of $i$
to the boundary in Theorem \ref{ct-free}. Note that the inclusion of $\Gamma$ into $\til H$ with its intrinsic metric is a quasi-isometry. So we might as
well replace the inclusion of $\Gamma$ into $\til M$ by that of $\til H$ into $\til M$. We shall show that
point pre-images of multiple points under $\partial i$
correspond to
end-points of leaves of an ending lamination in the Masur domain. 

The inclusion of $S$ into $H$ as its boundary induces a surjection of fundamental groups. Let $N$ denote the kernel. Let $S_N (=\partial \til{ H})$ 
denote the cover of $S$ corresponding to $N$. 

To distinguish between the ending lamination $\Lambda_{ELH}$ (in the Masur domain) 
and bi-infinite geodesics whose end-points are identified by $\partial i$, we make the following definition.

\begin{definition}
A {\bf CT leaf} $\lambda_{CT}$ is a bi-infinite geodesic whose end-points are identified by $\partial i$. \\ An {\bf EL leaf} $\lambda_{EL}$ is a bi-infinite geodesic whose end-points
are  ideal boundary points of
either a leaf of the ending lamination, or a complementary ideal polygon.
\end{definition}

We shall show that  \\
$\bullet$ {\bf An {\it EL leaf} is a {\it CT leaf}.}\\
$\bullet$ {\bf A {\it CT leaf} is an {\it EL leaf}.}

\begin{prop} {\bf  EL  is  CT } Let $G$ be a free degenerate Kleinian group without parabolics.
 Let $u,v$ be either ideal 
end-points of a leaf of an ending lamination of $G$, or ideal boundary points of a
complementary ideal polygon. Then
 $\partial i(u) = \partial i(v)$.
\label{ptpreimage}
\end{prop}

\begin{proof} This is almost identical to Proposition 3.1 of \cite{mahan-elct}. However, since the setup is
somewhat different we include a sketch of a proof.
 Take a sequence of short geodesics $\underline{s_i}$
 exiting the
end. Let $\underline{a_i}$ be geodesics in the intrinsic metric
on the  boundary $S$ (of $H$)
freely homotopic to $\underline{s_i}$. By topological tameness \cite{agol-tameness} \cite{gab-cal} and geometric tameness (\cite{thurstonnotes} Ch. 9) we may assume
 further that $\underline{a_i}$'s are simple closed curves on $S$. Join $\underline{a_i}$ to
 $\underline{s_i}$ by the shortest geodesic $\underline{t_i}$
 in $S \times [0, \infty )$ connecting the two
 curves. Then the collection $\underline{a_i}$ may be chosen to converge to the ending
 lamination on $S$ (\cite{thurstonnotes} Ch. 9). Also, in $S_N  \subset \til H \subset \til M$, we choose lifts
$a_i (\subset \til{S})$ (of $\underline{a_i}$)  
which are finite segments whose end-points are identified by the
 covering map $P: \til{S \times [0, \infty )} \rightarrow S \times [0, \infty )$. We also assume that $P$ is
 injective restricted to the interior of $a_i$'s mapping to
 $\underline{a_i}$.
 Similarly there
 exist 
 segments
$s_i \subset \til{M}$ 
which are finite segments whose end-points are identified by the
 covering map $P: \til{M} \rightarrow M$. We also assume that $P$ is
 injective restricted to the interior of $s_i$'s. 
The finite segments $s_i$ and $a_i$ are chosen in such
 a way that there exist 
 lifts $t_{1i}$, $t_{2i}$, joining end-points of $a_i$
 to corresponding end-points of $s_i$. The union of these four pieces looks like a trapezium (see below, where we have omitted subscripts for convenience).

\begin{center}

\includegraphics[height=4cm]{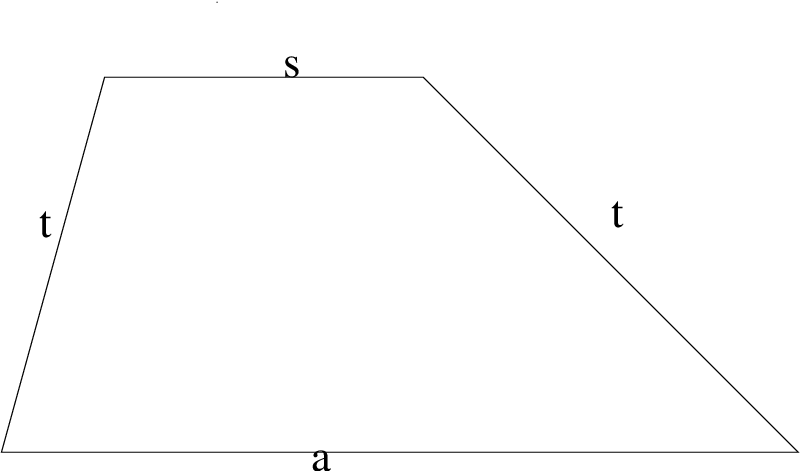}

\underline{Figure:  {\it Trapezium} }

\end{center}

\smallskip

Next, given any leaf $\lambda$ of the ending lamination, we may choose
translates of the finite segments $a_i$
(under the action of $\pi_1(H)$) appropriately, such that they
converge to $\lambda$ in $S_N$. For each $a_i$, let 

\begin{center}
$b_i = t_{1i} \circ s_i \circ {t_{2i}^{-1}}$
\end{center}
where ${t_{2i}^{-1}}$ denotes $t_{2i}$ with orientation
  reversed. If the translates of $a_i$ we are considering
  have end-points lying outside large balls around a fixed reference
  point $p \in S_N$, it is not hard to see that $b_i$'s lie
  outside large balls about $p$ in $\til{M}$. 
  
Here is a quick  sketch. 
We think of $\underline{t_i}$ as vertical and $a_i, s_i$ as horizontal.
We justify this heuristic.  We choose $\underline{t_i}$ to be a  distance minimizing geodesic  between a loop in
$S \times \{ 0 \}$ corresponding to $a_i$ and loop in
$S \times \{ i \}$ corresponding to $s_i$.
Saying that $\underline{t_i}$ is vertical
 is meant to convey the idea that $\underline{t_i}$ is indeed well-behaved with respect to the block
structure (coming from split geometry of the end) in terms of 
the progress it makes out the end. A (coarse) measure of the vertical co-ordinate of $z \in \underline{t_i}$
is the number of (thick or split) blocks that lie `below' $z$. 
The segments $t_{1i}, t_{2i}$ make {\it definite progress} out the end $S \times [0, \infty)$
in the sense that the number of blocks below $z$ (i.e. coarse vertical co-ordinate of $z$) is a {\bf proper} function
(say $ \Phi$) of the length of the initial segment of $\underline{t_i}$ up to $z$. Further, this proper function 
$ \Phi$ is
independent of the index $i$. This follows from the fact that a distance minimizing geodesic like $\underline{t_i}$
can spend only a bounded amount of time in each (split or thick) block. The bound depends on the block in general
but this is adequate for definite progress.

Next, 
$t_{1i}, t_{2i}$ are both lifts of $\underline{t_i}$.    Similarly
$s_i$ is contained in $ \til{S} \times \{ i \}$ and $a_i$
is contained in  $ \til{S} \times \{ 0 \}$ allowing us to legitimize the statement
that they are `horizontal'. Thus each $S_i$ and paths lying on them are being thought of as horizontal and paths
that make definite progress transversely are being thought of as vertical.  

Also 
  $\underline{s_i}$ lies at a large vertical height (at least $O(i)$)
from $H$ and hence $s_i$ lies at a large distance (at least $O(i)$) from $p$. Next, the initial
  point of $t_{1i}$ lies at a large horizontal distance $d_{hor}$ from $p$. Note that $d_{hor}$
may be chosen to be as large as we like (independent of $i$).
The vertical distance from $p$ increases monotonically (according to the proper function $\Phi$ independent of $i$)
as we proceed along 
   $t_{1i}$. The horizontal component (of a point $z$ on $t_{1i}$)
can go down at each step (as we move from the $m-$th block to the $(m+1)-$th block)
by a constant depending only on $m$. Since $d_{hor}$ may be chosen to be arbitrarily large,
it follows that for any point on $t_{1i}$ (or $t_{2i}$) at least one of the vertical or horizontal distances is large. 

\medskip

\noindent {\bf Note:} A parenthetical comment is in order.
 It is indeed possible (as was pointed out to us by the referee) to have a situation where
\begin{enumerate}
\item   geodesics $t_{1i}$ and $t_{2i}$ in $\mathbb{H}^3$  start and end outside large compacts 
\item and the geodesic $a_i$ also lies outside large compacts.
\item But the geodesic $\sigma_i$ joining the endpoint of $b_i$ do intersect a given compact $K_o$.
\end{enumerate}
To ensure this, the geodesics $t_{1i}$ and $t_{2i}$ must lie very close to $\sigma_i$ for a very long
period of time, and hence must themselves come close to $K_o$. This is prevented in our situation
by the definite progress of $\underline{t_i}$ (and hence $t_{1i}$ or $t_{2i}$).

\medskip

Returning to the proof,
we summarize  by saying  that 
   $t_{1i}$ (and similarly $t_{2i}$) lie at a  large  distance from $p$. Since $s_i$ is uniformly bounded in
length, it follows that
 $b_i$ lies at a  large  distance from $p$.

  Since ${\mathbb{H}}^3$
is $\delta$-hyperbolic for some $\delta > 0$, it follows that the geodesic joining the end-points of $b_i$ (and hence $a_i$ which has the same end-points) lies in a $2\delta -$ neighborhood
of $b_i$.

At this stage we invoke the existence theorem for Cannon-Thurston
maps, Theorem \ref{ct-free}. Since $a_i$'s converge to
$\lambda$ and the hyperbolic geodesics
 joining the end-points of $a_i$  exit all
compact sets, it follows that  $\partial i(u) = \partial i(v)$, where
$u, v$ denote the boundary points of $\lambda$. The Proposition
follows.
\end{proof}

\smallskip

\noindent {\bf Generalization to arbitrary finitely generated Kleinian groups:}
Any finitely generated Kleinian group is geometrically tame 
(\cite{agol-tameness} \cite{gab-cal} \cite{thurstonnotes} Ch. 9) and has finitely many ends. 
Observe that the proof of the above Proposition used the freeness of $G$ only at the stage of applying Theorem \ref{ct-free}. The same proof goes through
verbatim for freely decomposable Kleinian groups with degenerate ends. The only modification to the above proof is that we consider one end of the manifold $M$
at a time (and the pigeon-hole principle) along with Theorem \ref{ct} in place of Theorem \ref{ct-free}  to obtain the following Proposition.

\begin{prop} {\bf  EL  is  CT - General Case} Let $G$ be a finitely generated freely decomposable  Kleinian group.
 Let $u, v$ be either ideal 
end-points of a leaf of an ending lamination of $G$, or ideal boundary points of a
complementary ideal polygon. Then
 $\partial i(u) = \partial i(v)$.
\label{ptpreimage-genl}
\end{prop}

\subsection{CT leaves are EL leaves} As usual we deal first with  free
degenerate groups without parabolics.
We restate Theorem \ref{ptpre-inde} in a form that we shall use. Recall that
$M \setminus Int( H)$ is homeomorphic to
$S \times [0, \infty )$ and is bi-Lipschitz homeomorphic to an end $M_S$ of a simply degenerate hyperbolic manifold  without accidental parabolics
\cite{bowditch-model} \cite{minsky-elc3}. Hence by Theorem \ref{ptpre-inde} we have the following.

\begin{theorem} \cite{mahan-elct} Let $S, M_S$ be as above.
 Then the inclusion of universal covers
$\tilde{j} : \widetilde{S} \rightarrow \widetilde{M_S}$ 
extends
continuously to the boundary. Further, pre-images of points
on the boundary  are precisely ideal boundary points of a leaf of the ending
lamination $\Lambda_{ELS}$ of $M_S$, or ideal boundary points of
a complementary ideal polygon whenever the Cannon-Thurston map is not one-to-one.
\label{ptpre0}
\end{theorem}

We identify the Cayley graph $\Gamma$ of the free group with a subset of $ \til H \subset \til M$, viz. the orbit of a base-point joined by edges. The
next Theorem is one of the main Theorems of this paper.

\begin{theorem}  Let $G$ be a free degenerate  Kleinian group without parabolics. 
 Let $i : \Gamma_G \rightarrow {\mathbb H}^3$ be the natural identification
of a Cayley graph of $G$ with the orbit of a point in ${\mathbb H}^3$. Then $i$ extends continuously to a map 
$\hat{i}: \hhat{\Gamma_G} \rightarrow {\mathbb D}^3$,
where $\hhat{\Gamma_G}$ denotes the (Gromov) hyperbolic compactification of $\Gamma_G$. Let $\partial i$ denote the restriction of $\hat{i}$ to the boundary $\partial \Gamma_G$ of $\Gamma_G$.

 Then $\partial i(a) = \partial i(b)$ for $a\neq b \in \partial \Gamma$ iff $a, b$ are either ideal end-points of a leaf of an ending lamination of $G$, or ideal boundary points
of a complementary ideal polygon.
\label{ptpreimagefinal}
\end{theorem}

\begin{proof}
By Theorem \ref{ct-free} the inclusion $i: \Gamma \rightarrow \til M$ extends continuously
to a map between the Gromov compactifications $\hat{i} : \widehat{\Gamma} \rightarrow {\mathbb{D}}^3$.
Let $\partial i$ denote the values of the above continuous extension to the boundary. Suppose $\partial{i} (a) = \partial{i} (b)$.  $\Lambda_{EL}$ is
the ending  lamination of $M$ regarded as a subset of $S$.  Let $\Lambda_{ELG}$ denote $\Lambda_{EL}$ lifted to $S_G = \partial {\til H}$, which is a cover
of $S$.
We want to show
that $a, b$ are the end-points of a leaf of $\Lambda_{ELG}$. Suppose $(a,b)_\Gamma$ is the bi-infinite geodesic from $a$ to $b$ in $\Gamma
\subset \til M$. Assume without loss
of generality that $(a,b)$ passes through $1 \in \Gamma$. Let $a_k, b_k$ be points on  $(a,b)$ such that
 $a_k \rightarrow a$ and $b_k \rightarrow b$. Let $\overline{a_kb_k}$ denote 
the geodesic in $\til M$ joining $a_k, b_k$. By continuity of the Cannon-Thurston map (Theorem \ref{ct-free}) there exists $N(k) \rightarrow \infty$
as $k \rightarrow \infty$ such that $\overline{a_kb_k}$ lies outside an $N(k)$ ball about $1 \in \Gamma \subset \til M$, where radius is measured
in the hyperbolic metric on $ \til M$. Isotoping $\overline{a_kb_k}$ slightly, we can assume without loss
of generality  that it meets $\Gamma \subset \til M$ only at its end-points
(since $\Gamma$ is one dimensional). We can further isotope $\overline{a_kb_k}$ rel. endpoints by a bounded amount (depending on the Hausdorff distance
between $\til H$ and $\Gamma \subset \til{H}$) such that $\overline{a_kb_k}\cap S_G =\{ c_k, d_k \}$, where.\\
1)  $c_k, d_k \in S_G = \partial {\til H}$ with $d(a_k, c_k)$ and $d(b_k, d_k)$
 uniformly bounded (independent of $k$) \\
2) if $\overline{c_kd_k}$
denotes the subpath of $\overline{a_kb_k}$ between $c_k, d_k$ then (modifying $N(k)$ by an additive constant if necessary) $\overline{c_kd_k}$
lies outside an $N(k)$ ball about $1 \in \Gamma \subset \til M$.\\
3) $\overline{c_kd_k}$ intersects $\til H$ only at the endpoints $c_k, d_k$. \\

Thus $\overline{a_kb_k}$ is a concatenation of three pieces, $\overline{a_kc_k}$, $\overline{c_kd_k}$, $\overline{d_kb_k}$, where $\overline{a_kc_k}$ and 
$\overline{d_kb_k}$ are uniformly bounded in length and lie in $\til H$, whereas  $\overline{c_kd_k}$ lies in $\til{M} \setminus Int(\til{H})$.

 Let $[c_k,d_k]_{G}$ denote the geodesic in the intrinsic metric
on $S_G$ which is  homotopic (rel. endpoints)  to $\overline{c_kd_k}$ in $\til M \setminus Int(\til{H})$. Since $G$ is free, we can assume that its Cayley graph is a tree
and (since $\til H$ is quasi-isometric to $\Gamma$)  $[c_k,d_k]_{G}$ passes through a point $o_k \in S_G$ at a uniformly bounded distance from $1 \in \Gamma$.

Recall that $\til{S} (\subset \til{M_S})$ is the universal cover of $S$ inside the universal cover of the end $M_S$.
Lift $[c_k,d_k]_{G}$ to some geodesic $\til{c_k,d_k} (\subset \til{S})$ in the intrinsic metric on $\til{S}$. Further assume that there exists
some fixed $o \in \til S$ such that the corresponding lift $o_k^{\prime}$ of $o_k$ lies in a uniformly bounded neighborhood of $o$. Let 
$({\overline{c_kd_k}})^\sim$ denote the corresponding lift of  $\overline{c_kd_k}$ having the same endpoints as $\til{c_k,d_k}$ (such a choice is possible as
 $[c_k,d_k]_{G}$ and $\overline{c_kd_k}$ are homotopic rel. endpoints in the complement of $ Int(\til{H})$ in $\til M$). It follows that  
$({\overline{c_kd_k}})^\sim$ lies outside an $N(k)$-ball about $o_k^{\prime}$ in $\til{M_S}$. Hence (modifying $N(k)$ by a further additive constant if
necessary), $({\overline{c_kd_k}})^\sim$ lies outside an $N(k)$-ball about $o \in \til{M_S}$. Therefore, by the existence of Cannon-Thurston maps
for $j: \til S \rightarrow \til{M_S}$ (Theorem \ref{ptpre0}) it follows that if $\til{c_\infty d_\infty}$ denotes any subsequential
limit of the segments $\til{c_k,d_k}$ on $\til S$, then $\partial j (c_\infty) = \partial j (d_\infty) $ and hence again by 
Theorem \ref{ptpre0} $c_\infty , d_\infty$ are end-points of leaves (or vertices of a complementary ideal polygon)
of the ending lamination  $\Lambda_{ELS}$ of $S \subset M_S$. Finally, since
$\til{c_\infty , d_\infty}$ is a  bi-infinite geodesic passing through a bounded neighborhood of $o$, it projects to a  leaf (or diagonal of
a complementary ideal polygon) of $\Lambda_{ELG}$
in $S_G$. Such leaves are also well-defined as leaves of the ending lamination $\Lambda_{ELH}$, i.e. as leaves of the ending lamination of $M$ regarded
as an element of the Masur domain, cf. Theorem \ref{elinmd}. We have thus finally shown that $\Lambda_{CT} \subset \Lambda_{ELH}$. Combining this with Proposition \ref{ptpreimage}
and Theorem \ref{ct-free}
we have the Theorem. 
\end{proof}

Note that in the proof of Theorem \ref{ptpreimagefinal} we have used freeness of $G$ to conclude only two things: \\
1) The manifold $M$ has exactly one end. \\
2) The path $\lambda$ in $\til H$ can be isotoped off  the Cayley graph of $G$ embedded in $\til H$. \\

To prove an analogue of Theorem \ref{ptpreimagefinal} for arbitrary finitely generated Kleinian groups we continue with
the notation that $M$ is a hyperbolic manifold with augmented Scott core $H$. Then $M$
has finitely many ends. We first note, that if $\lambda =(a_\infty , b_\infty )$ is a CT leaf then there exist $a_n \rightarrow a_\infty$
and $b_n \rightarrow b_\infty$ such that the geodesic realizations $\mu_n$ of $[a_n,b_n]$ in $\til M$ leave arbitrarily large compact sets. 
We may assume that $\til{M} \setminus \til{ H}$ consists of lifts of the ends of $M$ to $\til M$. Each $\mu_n$ intersects finitely many such
lifts of ends and hence has subsegments $\mu_{n1}, \cdots \mu_{nk}$, where each $\mu_{ni}$ lies in a lift $\til{H\cup E_i}$ for some end $E_i$
of $M$. Assume that such a decomposition is minimal (i.e. $k$ is the minimal possible for $\mu_n$). Then, since
$\mu_n$ leaves arbitrarily large compact sets, so must each $\mu_{ni}$. Further, each $\mu_{ni}$ has end-points
in $\til H$. It follows that there will be at least one of the $\mu_{ni}$'s - call it $\nu_n$ -  such that \\
\begin{enumerate}
\item[(a)] $\nu_n$ is contained entirely in one of these lifts of the ends \\
\item[(b)] Endpoints $c_n, d_n$ of $\nu_n$ lie on $\til H$\\
\item[(c)] $c_n \rightarrow c_\infty$
and $d_n \rightarrow d_\infty$, where $c_\infty, d_\infty $ lie in the boundary of $\til H$. \\
\item[(d)] Finally, by considering all  segments $\nu_n$ (the non-uniqueness of $\nu_n$ is used at this stage)
satisfying properties (1)-(3), there exists a finite sequence $a_\infty = a_0, \cdots , a_n=b_\infty$
such that each pair $(a_i, a_{i+1})$ arises as a limiting pair $c_\infty, d_\infty $ as in (3). 
\end{enumerate}
 
 We may therefore assume for the time being that $\mu_n$ lies in precisely one of the lifts of the ends $E$ of $M$. If $S^h = H \cap E$ be its boundary 
then the ending lamination lies in the boundary of the (relatively)
hyperbolic group $j_\ast (\pi_1(S^h))$ (hyperbolic relative to the cusp groups if any), where $j: S^h \rightarrow M$ is inclusion.

Fact (2) now  goes through for arbitrary finitely generated Kleinian groups, as the inclusion of the augmented Scott core into $M$ is a 
homotopy equivalence (in fact a deformation retract)
and we are only interested in leaves which are limits of segments whose geodesic realizations lie inside the lift of a fixed end.

With this modification, and with Theorem \ref{ct} in place,
the proof of Theorem \ref{ptpreimagefinal} goes through for arbitrary finitely generated Kleinian groups
provided a Minsky model obtains (for instance if $M$ has no parabolics, cf. Appendix \ref{app}). However, owing to Item (d) above,
 the statement is a bit more involved.

\begin{theorem}  Let $G$ be a finitely generated Kleinian group. Let $M= \Hyp^3/G$ and assume
 that each degenerate end of $M$ admits a Minsky model  (for instance if $M$ has no parabolics).
 Let $i : \Gamma_G \rightarrow {\mathbb H}^3$ be the natural identification
of a Cayley graph of $G$ with the orbit of a point in ${\mathbb H}^3$. Then $i$ extends continuously to a map 
$\hat{i}: \hhat{\Gamma_G} \rightarrow {\mathbb D}^3$,
where $\hhat{\Gamma_G}$ denotes the (relative) hyperbolic compactification of $\Gamma_G$. Let $\partial i$ denote the restriction of $\hat{i}$ to the boundary $\partial \Gamma_G$ of $\Gamma_G$.

Let $E$ be a degenerate end of $N^h= {\mathbb H}^3/G$ and $\til E$ a lift of $E$ to $\til{N^h}$
and let $M_{gf}$ be an augmented Scott core of $N^h$. Then the ending lamination $\LL_E$ for the end $E$ lifts to a lamination
on $\til{M_{gf}} \cap \til{E}$. Each such lift $\LL$ of the ending lamination of a degenerate end defines a relation $\RR_\LL$ on the (Gromov) 
hyperbolic boundary $ \partial \widetilde{M_{gf}}$
(equal to the relative hyperbolic boundary $\partial \Gamma_G$ of $\Gamma_G$),
 given by
$a\RR_\LL b$ iff $a, b$ are  end-points of a leaf of $\LL$. Let $\{ \RR_i \}_i$ be the entire collection of relations on $ \partial \widetilde{M_{gf}}$ obtained this way. Let $\RR$ be the transitive closure of
the union $\bigcup_i \RR_i$. Then $\partial i(a) = \partial i(b)$ iff $a\RR b$.
\label{ptpreimagefinal-kg}
\end{theorem}

\section{Applications} 
In this section we shall first mention a couple of  applications  of the main Theorems of this paper. Finally we indicate an extension of the Sullivan-McMullen dictionary between complex
dynamics and Kleinian groups.

\subsection{Primitive Stable Representations} In  \cite{minsky-primitive} Minsky
 introduced and studied primitive stable representations, an open set of
$PSl_2({\mathbb C})$ characters of a nonabelian free group, on which the action of
the outer automorphism group is properly discontinuous, and which is strictly larger
than the set of discrete, faithful convex-cocompact  (i.e. Schottky)
characters.

In  \cite{minsky-primitive} Minsky also conjectured that \\
{\it A discrete faithful representation of $F$  is
primitive-stable if and only if every component of the ending lamination is blocking.}

Using the structure of the Cannon-Thurston map for handlebody groups,   Jeon, Kim, Lecuire and Ohshika
 \cite{woojin} have solved this conjecture. We sketch their proof  for degenerate free groups without parabolics with associated representation
 denoted by $\rho$. 
 They  show that
for the ending lamination $\Lambda_E$ of a degenerate free group
without parabolics, $Wh(\Lambda_E , \Delta)$  is connected
and has no cutpoints.
  Then they argue by contradiction. 
 If  $\rho$  is not
primitive stable, then there exists a sequence of primitive cyclically reduced elements $w_n$ such
that $\rho(\cdots w_nw_nw_n \cdots)$ is not an $n-$ quasi-geodesic. 
After passing to a subsequence, $w_n$ and hence $\cdots w_nw_nw_n \cdots$ converges to a bi-infinite geodesic $ w_\infty$ in the Cayley graph with
two distinct end points $w_+, w_-$  in the Gromov boundary of $F$ identified by the Cannon-Thurston map.
It follows from a Lemma due to  Whitehead  that 
$w_n$ cannot be primitive for large $n$, a contradiction.

\subsection{Discreteness of Commensurators}
In \cite{llr} and \cite{mahan-commens}, the main results of this paper are used to prove that commensurators of finitely generated,  infinite covolume, Zariski dense Kleinian
groups are discrete. The proof proceeds by showing that commensurators preserve the structure of point pre-images of Cannon-Thurston maps.

\subsection{Extending the Sullivan-McMullen Dictionary }

A celebrated theorem of Yoccoz in Complex Dynamics (see Hubbard
\cite{hubbard-yoccoz}, or Milnor
\cite{milnor-yoccoz}) proves the local connectivity of certain Julia
sets using a technique called `puzzle pieces', which consists of a
decomposition of a complex domain into pieces each of which under
iteration by a quadratic map converges to a single point. The
dynamical system can then be regarded as a semigroup ${\mathbb{Z}}_+$
of transformations acting on a complex domain. 

Split  components can be regarded as a 3-dimensional analogue
of puzzle pieces. Let us try to justify this analogy.
Suppose there is a group $G$ acting cocompactly on
${\mathbb{H}}^3$. Let $H \subset G$ be a subgroup. Let $G/H$ denote the coset space. 
Then what we would want as the right analogue is that if one takes a sequence of elements
$g_i$ going to infinity in the
coset space, the iterates of the convex hull of the limit set of $H$ converge to a point
in the limit sphere. Thus going to infinity in the coset space  $G/H$ would be the right Kleinian groups analogue of 
 going to infinity in the  semigroup ${\mathbb{Z}}_+$
of transformations acting on a complex domain.

In the context of this paper, $H$ would be the fundamental group of a split component.
 However, for us
 $G$ is a surface Kleinian group and does not act co-compactly on ${\mathbb{H}}^3$. We think of the quotient space ${\mathbb{H}}^3/H$ as parametrizing
the set of normal directions to the split
component. The graph
metric gives a combinatorial distance on  ${\mathbb{H}}^3/H$ and we think of $({\mathbb{H}}^3/H, d_G)$ as the analogue of  the  semigroup ${\mathbb{Z}}_+$. 
 Thus, instead of going to
infinity by iteration in  the  semigroup ${\mathbb{Z}}_+$, we go to infinity in the graph metric. Further,
the analogue of the  requirement that iterates go to infinity, is that
the visual diameter goes to zero as we move to infinity in the graph
metric. This is  ensured by hyperbolic quasiconvexity, and also
follows easily from {\bf graph quasiconvexity}. Note that {\bf graph
quasiconvexity} is a statement that gives uniform shrinking of visual
diameter to zero as one goes to infinity.

Thus we extend the Sullivan-McMullen dictionary (see
\cite{sullivan-dict}, \cite{ctm-renorm}) between Kleinian
groups and complex dynamics by suggesting the following analogy:

\begin{enumerate}
\item {\em Puzzle pieces} are analogous to {\em split components} \\
\item {\em Convergence  to a point under iteration} is analogous to
  {\em graph quasiconvexity} \\
\end{enumerate}

One issue that gets clarified by the above analogy is a point raised
by McMullen in \cite{ctm-locconn}. McMullen indicates that though the
Julia set $J(P_\theta )$, where 
\begin{center}
$P_\theta (z) = e^{2\pi i \theta}z + z^2$
\end{center}

\noindent need not be locally connected in general by a result of Sullivan
\cite{sullivan-lc}, the limit sets of  punctured torus groups are
nevertheless locally connected. Local connectivity of Julia sets would therefore not be the right analogue of 
local connectivity of limit sets in this setup. Instead we look at the {\it techniques} for proving local connectivity of limit sets 
vis-a-vis  the {\it techniques} for proving local connectivity of Julia sets. Thus,  by proposing the analogy between puzzle
pieces and split components as above, this issue is to an extent clarified. In short, the analogy is in the technique rather than in the result.

\smallskip

An analogue of the ${\mathbb{Z}}_+$ dynamical system may also be
extracted from the split geometry model. Note that each block
corresponds to a splitting of the surface group, and hence an action
on a tree. As $i \rightarrow \infty$, the split blocks $B^s_i$ and
hence the induced splittings also go to infinity, converging to a {\bf
free
action of the surface group on an ${\mathbb{R}}$-tree dual to the ending
lamination}. Thus iteration of the quadratic function corresponds to
taking a sequence of splittings of the surface group converging to a
(particular)
action on an ${\mathbb{R}}$-tree. 

\smallskip

{\bf Problem:} The building of the Minsky model and its bi-Lipschitz
equivalence to a hyperbolic manifold \cite{minsky-elc1}
\cite{minsky-elc2} gives rise to a speculation that there should be a
purely combinatorial way of doing much of the work. Bowditch's
rendering \cite{bowditch-endinv}, \cite{bowditch-model}
of the Minsky, Brock-Canary-Minsky results is a step in this
direction. This paper brings out the possibility that the whole thing
should be do-able purely in terms of actions on trees. Of course there
is an action of the surface group on a tree dual to a pants
decomposition. So we do have a starting point. However, one ought to
be able to give a purely combinatorial description, {\em ab initio},
in terms of a sequence of actions of surface groups on trees
converging to an action on an ${\mathbb{R}}$-tree. This would open up the
possibility of extending these results (including those of this paper)
to other hyperbolic groups with
infinite automorphism groups, notably free groups.

\section{Appendix: Model Manifold for compressible boundary} \label{app} In this Appendix, for the sake of completeness,
 we provide a sketch
of a bi-Lipschitz Minsky model for general $M$ without parabolics, following work and ideas
of several authors, notably Brock, Bromberg and Souto. No claim to originality is made here.
Our aim is modest:
to reduce the problem of existence of such a model to the case with incompressible boundary discussed in \cite{minsky-elc2} and hence conclude the existence of a bi-Lipschitz model geometry in general. 
We shall focus on a degenerate end $E$ such that $\partial E = S$
is compressible. We would like to prove that $E$ has a bi-Lipschitz model as in \cite{minsky-elc2}.
By passing to the cover of $M$ corresponding to the end $E$ 
(i.e. the cover of $M$ corresponding to the image of $\pi_1(E) $ in $\pi_1(M)$),
we might as well {\bf assume that $M$ is a compression body}. We sketch the proof below in the
case when $S (= \partial E)$ is closed.
 The case when $S$ has parabolics, and hence $M$ has rank one cusps, is technically quite a bit more
more involved and we refer the reader to \cite[p. 148]{BBES} 
where the presence of accidental parabolics 
is addressed.

The model geometry for the case where $E$ has bounded geometry has been treated in detail in 
\cite{mosher-model, bowditch-stacks, ohshika-tams} (see also \cite{mitra-trees})
and so we assume that $E$ has unbounded geometry. The main idea is to isolate $E$ and reduce to
the incompressible case. The point is to show that the asymptotic geometry of $E$ is bi-Lipschitz to
an incompressible simply degenerate end.\\

\noindent {\bf Step  1: Existence of Disk-busting curves:}  We consider a sequence of simple closed curves $c_i$ on $S$
whose geodesic realizations  exit
$E$ and have length going to zero. We note first that for large enough $i$, these curves
nontrivially intersect all compressing disks with boundary in $S$. If not, there exists compressing disk $D_i$
with $\partial D_i \subset S$ such that $ \partial D_i \cap c_i = \emptyset$. Normalizing $c_i$ by its length
and taking a limit in $\PML(S)$, it would follow that the (limiting) ending lamination does not lie in the Masur Domain, a contradiction (cf.
the definition of the Masur domain after Theorem \ref{free-model}).
 Hence, for all large enough $i$, the curves $c_i$ are disk-busting.

Hence if we drill out the
geodesic realizations of $c_i$ from $E$, and equip the resulting manifold  with a complete hyperbolic structure
with a rank 2 cusp where $c_i$ has been drilled
(while fixing the end-invariants), the new end $E_i$ is incompressible.  Then, by \cite{minsky-elc2},
$E_i$ has a model geometry. 
We would be done if we could establish that $E$ and $E_i$ are bi-Lipschitz 
homeomorphic, as this would transfer the bi-Lipschitz model of $E_i$ to $E$. Instead of drilling, we shall
perform the closely related construction of grafting below.\\

\noindent {\bf Step 2: Grafting Constructions:} In
 \cite{bromberg-proj, brock-bromberg-density, BBES,
brock-bromberg-jtop}, two kinds of  grafting constructions are used. 
We call these up-grafting and down-grafting below.

The standard construction  \cite{bromberg-proj} (Section 4, where grafting an infinite annulus is described)
involves cutting the manifold open along a semi-infinite 
cylinder and gluing in a wedge. Thus one {\it up-grafts} an annulus along the geodesic realization of $c_i$
in $M$ as in \cite{bromberg-proj, brock-bromberg-density} to get a cone manifold $N_i$
with cone angle $4\pi$ along the geodesic realizations of $c_i$.

The other kind of grafting (which we call down-grafting) is
described in detail in \cite{bromberg-proj} (see Section 4 of the paper, especially 
Theorem 4.2)\footnote{I am grateful to Ken Bromberg for explaining this
construction to me.}. We give a quick summary.     
 Take an annulus going out the end based at the curve $c_0$.
 We pass to the cover of $E$ associated to the subsurface obtained by taking the complement of the curve. For 
ease of exposition, assume that $c_0$
 is non-separating.
 In the cover the  annulus will have two isometric lifts. Cut the cover
along both of them and then glue them together. The result is
 a cone-manifold $M_{c_0}$ that is homeomorphic to  $S \times \R$.

To ensure that (up or down) 
grafting is possible, one needs two conditions: \\
a) that the geodesic is unknotted, \\
b) that the geodesic is sufficiently small. This allows the relevant
technical tools of \cite{bromberg-proj}, \cite{brock-bromberg-density} (Section 5, especially 
Theorem 5.1 which proves the existence of  the geometrically finite
cone manifold), 
  to go through.\\

The first condition follows from the second thanks to work of Otal
\cite{otal-unknot, otal-unknot2} (See also \cite{souto-unknot}). Shortness of curves is guaranteed
by unbounded geometry of $E$.\\

\noindent {\bf Step 3: Double Grafting and Convergence:}
Choose sufficiently short (short enough to allow grafting in Step 2 to go through)
disk-busting curves $c_0$ and $c_i$ ($i\in \natls$). 
Along $c_0$ perform down-grafting to obtain 
a cone-manifold $M_{c_0}$ that is homeomorphic to  $S \times \R$. As a result, 
one end (the “up” end) will be the original unbounded geometry end. The “down” end will be geometrically finite.
This process effectively isolates the "up" end (isometric to $E$) with the other end of
 $M_{c_0}$ being geometrically finite. We need to show that $E$ is bi-Lipschitz to the
 end of a simply degenerate manifold.

Along the $c_i$ curves now perform the standard (up)
grafting construction on $M_{c_0}$ described in Step 2 above to obtain quasi-Fuchsian 
cone manifolds $N_i$. Smooth quasi-Fuchsian 
 manifolds
$M_i$ may be obtained from $N_i$ by decreasing the cone angle to $2\pi$ \cite{bromberg-proj}.

The main technical Theorems of Brock and Bromberg on 
 inflexibility
  \cite[Theorem 3.6]{brock-bromberg-jtop}, \cite[Theorems 1.1, 1.2]{brock-bromberg-inflex} now guarantee
 that the cone deformations 
do not change the hyperbolic structure much away from the boundary. More precisely,
fix an arbitrary compact core
$\MM$ of $M_{c_0}$
 and assume that all the curves $c_i$ have geodesic realizations outside $\MM$.  Then the inflexibility
Theorem \cite[Theorem 3.6]{brock-bromberg-jtop} shows that the cone deformation from $N_i$ to $M_i$
has small effect in the thick part of $\MM$ (more precisely, the harmonic strain field decays exponentially
with distance from the boundary). It follows that the hyperbolic structures on $N_i$ and $M_i$ 
 when restricted to $\MM$ are
$(1+\epsilon_i)-$biLipschitz to each other, where $\epsilon_i \to 0$ as $i \to \infty$.

Now, in the cone-manifolds $N_i$, between the two curves
$c_0$ and $c_i$, there is a larger and larger product region that is 
isometric
to a product region in the original end
$E$. Therefore (from the inflexibility theorem above) the sequence of smooth quasifuchsian manifolds
$M_i$ converges to a simply degenerate manifold
$M_\infty$ whose  degenerate end 
$E_\infty$ is bilipschitz to the original end
$E$. Since $E_\infty$ has a model geometry by \cite{minsky-elc1, minsky-elc2}, so does $E$.

\bibliography{hbody}
\bibliographystyle{alpha}

\end{document}